\documentclass[11pt, a4paper, UKenglish]{article}

\usepackage[UKenglish]{babel}
\usepackage{MA_Titlepage}

\usepackage{amsmath}
\usepackage{amssymb}
\usepackage{amsthm}
\usepackage{bbm}
\usepackage{hyperref}
\usepackage{cleveref}
\usepackage{mathtools}
\usepackage{pgfplots}
\usepackage{esint}
\usetikzlibrary{math}
\pgfplotsset{width=5cm,compat=1.9}
\newtheorem{theorem}{Theorem}[section]
\newtheorem{lemma}[theorem]{Lemma}
\newtheorem{corollary}{Corollary}[theorem]
\theoremstyle{definition}
\newtheorem{definition}[theorem]{Definition}
\DeclareMathOperator{\rank}{rank}
\DeclareMathOperator{\sgn}{sgn}
\DeclareMathOperator{\conv}{conv}

\DeclareMathOperator{\ReLU}{ReLU}
\DeclareMathOperator{\Lip}{Lip}
\DeclareMathOperator{\I}{Im}
\DeclareMathOperator{\R}{Re}

\usepackage[backend=biber, style=alphabetic]{biblatex}
\usepackage{csquotes}

\newbibmacro{string+doiurl}[1]{%
  \iffieldundef{doi}{%
    \iffieldundef{url}{%
         #1%
      }{%
      \href{\thefield{url}}{#1}%
    }%
  }{%
    \href{https://doi.org/\thefield{doi}}{#1}%
  }%
}
\DeclareFieldFormat[article,misc,thesis,incollection,online,book]{title}{\usebibmacro{string+doiurl}{\mkbibquote{#1}}}

\AtEveryBibitem{%
  \clearfield{month}%
  \clearfield{eprintclass}%
}

\authornew{Jan Holstermann}
\geburtsdatum{8th March 2000}
\geburtsort{Ahlen, Germany}
\date{\today}

\betreuer{Advisor: Prof. Dr. Christoph Thiele}
\zweitgutachter{Second Advisor: Dr. Rajula Srivastava}

\institut{Mathematisches Institut}
\title{On the Expressive Power of Neural Networks}
\ausarbeitungstyp{Master's Thesis  Mathematics}

\begin{document}

\maketitle

\clearpage

\begin{abstract}
In 1989 George Cybenko proved in a landmark paper that wide shallow neural networks can approximate arbitrary continuous functions on a compact set \cite{univAppr}. This universal approximation theorem sparked a lot of follow-up research. \\

Shen, Yang and Zhang determined optimal approximation rates for ReLU-networks in $L^p$-norms with $p \in [1,\infty)$ \cite{optimalApproximationRate}. Kidger and Lyons proved a universal approximation theorem for deep narrow ReLU-networks \cite{universalApproximation}. Telgarsky gave an example of a deep narrow ReLU-network that cannot be approximated by a wide shallow ReLU-network unless it has exponentially many neurons \cite{widthInefficiency}. \\

However, there are even more questions that still remain unresolved. Are there any wide shallow ReLU-networks that cannot be approximated well by deep narrow ReLU-networks \cite{widthDepthQuestion}? Is the universal approximation theorem still true for other norms like the Sobolev norm $W^{1,1}$? Do these results hold for activation functions other than ReLU? \\

We will answer all of those questions and more with a framework of two expressive powers. The first one is well-known and counts the maximal number of linear regions of a function calculated by a ReLU-network. We will improve the best known bounds for this expressive power. The second one is entirely new.
\end{abstract}

\newpage

\tableofcontents

\newpage

\section{Introduction}

The goal of this thesis is to prove various structural results about neural networks using so-called expressive powers. More precisely, we will consider feedforward neural networks, which are defined in the following way.

\begin{definition}[\cite{differentLowerBoundRn}, Definition 1]
Let $\rho: \mathbb{R} \to \mathbb{R}$, $L \in \mathbb{N}$ and 
\begin{align*}
    \mathbf{n} = (n_0,n_1,\dots,n_{L+1}) \in \mathbb{N}^{L+2}.
\end{align*}
A \textit{feedforward neural network} with design $\mathbf{n}$ and activation function $\rho$ calculates the function $f: \mathbb{R}^{n_0} \to \mathbb{R}^{n_{L+1}}$ defined by
\begin{align*}
    f \coloneqq f_L \circ \rho \circ f_{L-1} \circ \rho \circ \cdots \circ f_1 \circ \rho \circ f_0,
\end{align*}
where $f_i: \mathbb{R}^{n_i} \to \mathbb{R}^{n_{i+1}}$ for $i = 0,\dots,L$ is an affine function and $\rho$ is used component-wise. This neural network consists of $L$ \textit{hidden layers}, where the layers are defined as $\rho \circ f_i$ for $i = 0,\dots,L-1$. The final affine function $f_L$ is called the \textit{output layer}.
\end{definition}

Hence, a neural network is the composition of its layers. We can further decompose layers into neurons.

\begin{definition}[\cite{universalApproximation}, Section 4]
A \textit{neuron} is defined as the composition of an activation function $\rho$ and an affine function $f: \mathbb{R}^{n} \to \mathbb{R}$ with $n \in \mathbb{N}$.
\end{definition}

This means that we can think of the $i$th hidden layer as a collection of $n_i$ neurons. Most of the time we will use $\ReLU(x) \coloneqq \max\{0,x\}$ as the activation function. Any feedforward neural network with this activation function will compute a continuous function. We are particularly interested in the space of functions that can be calculated.

\begin{definition}[\cite{lowerBoundRn}, Definition 1]
Let $L \in \mathbb{N}$ and 
\begin{align*}
\mathbf{n} = (n_0,n_1,\dots,n_{L+1}) \in \mathbb{N}^{L+2}.
\end{align*}
Define $\mathcal{F}_{\mathbf{n}}$ to be the set of functions $\mathbb{R}^{n_0} \to \mathbb{R}^{n_{L+1}}$ that can be computed by a feedforward neural network of design $\mathbf{n}$.
\end{definition}

While it seems unfeasible to describe any possible function in $\mathcal{F}_{\mathbf{n}}$ in a neat way, we can focus on attributes of those functions to get some insight. Consider a neural network of design $\mathbf{n}$ with ReLU as its activation function. One important characteristic of the members of $\mathcal{F}_{\mathbf{n}}$ is that they are composed of linear regions.

\begin{definition}[\cite{lowerBoundRn}, Definition 1]
An open connected subset $R \subseteq \mathbb{R}^{n_0}$ is called \textit{linear region} of $f: \mathbb{R}^{n_0} \to \mathbb{R}^{n_{L+1}}$ if $f|_R$ is affine and $f|_{\Tilde{R}}$ is non-affine for every open set $\Tilde{R} \supsetneq R$.
\end{definition}

These linear regions fit together to form a piecewise affine function, which is calculated by the neural network.

\begin{definition}[\cite{defPiecewiseAffine}, Definition 2.1]
We call a function $f: \mathbb{R}^{n_0} \to \mathbb{R}^{n_{L+1}}$ \textit{piecewise affine} if $f$ is continuous,
\begin{align*}
    M \coloneqq \{R \subseteq \mathbb{R}^{n_0}: R \text{ linear region of } f\}
\end{align*}
is finite and
\begin{align*}
    \mathbb{R}^{n_0} = \bigcup_{R \in M} \overline{R}.
\end{align*}
\end{definition}

Every function calculated by a ReLU-network is a piecewise affine function simply because compositions of piecewise affine functions are again piecewise affine. To get a sense of how complex a piecewise affine function is, one can for example count the number of linear regions. Generally speaking, more linear regions give rise to more complex functions. This idea is encapsulated in the definition of our first expressive power.

\begin{definition}[\cite{lowerBoundRn}, Definition 1]
We define the \textit{expressive power} $\mathcal{R}(\mathbf{n})$ as the maximal number of linear regions of a function in $\mathcal{F}_{\mathbf{n}}$.
\end{definition}

Next up, we will show that for the study of $\mathcal{R}(\mathbf{n})$ it is enough to consider $n_{L+1} = 1$. This is also stated in Section $2$ of \cite{upperBoundRn} but with a different proof. 

\begin{theorem}[\cite{upperBoundRn}, Section 2]
\label{theorem: n_L+1=1}
We have
\begin{align*}
\mathcal{R}(n_0,\dots,n_L,n_{L+1}) = \mathcal{R}(n_0,\dots,n_L,1).
\end{align*}
\end{theorem}

This gives us everything we need to find some values for $\mathcal{R}(\mathbf{n})$. We will start with the expressive power in the case $n_0 = 1$. In contrast to the general case, we can get an exact formula here. In \cite{upperBoundRn} it was proven that in case of $n_i \geq 3$ for $1 \leq i \leq L$ we have
\begin{equation*}
    \mathcal{R}(\mathbf{n}) = \prod_{i=1}^{L} (n_i + 1).
\end{equation*}
We will provide a slight generalization to $n_i \geq 2$.

\begin{theorem}[generalization of \cite{upperBoundRn}, Theorem 7]
\label{theorem: R(n) in 1 dimension}
Consider a neural network with $n_0 = 1$ and $n_i \geq 2$ for $i = 1,2,\dots,L$. Then we have
\begin{equation*}
    \mathcal{R}(\mathbf{n}) = 1 + \sum_{i=1}^L n_i \prod_{j=1}^{i-1} (n_j + \mathbbm{1}_{n_j > 2}).
\end{equation*}
\end{theorem}

Observe that for $n_i \geq 3$ this expression simplifies to the previous expression in \cite{upperBoundRn}. Moreover, we can draw two interesting conclusions from this result, which are stated below. The proofs can be found in \Cref{Section 2}. The first conclusion is that layers with $2$ neurons can form a bottleneck.

\begin{corollary}[Bottleneck Effect]
\label{corollary: R(n) in 1 dimension bottleneck}
Swapping a hidden layer with two neurons to the end of the network can only increase the expressive power. The order of the layers with $\geq 3$ neurons is irrelevant as long as they are before the layers with $2$ neurons.
\end{corollary}

The second conclusion is that we can even determine an optimal network architecture based on this expressive power.

\begin{corollary}[Optimal Network Architecture]
\label{corollary: R(n) in 1 dimension optimal network architecture}
For a fixed number of neurons the network consisting of several hidden layers with $3$ neurons followed by $0$, $1$ or $2$ hidden layers with $2$ neurons maximizes the expressive power $\mathcal{R}$.
\end{corollary}

Unfortunately, we do not have an exact formula for arbitrary $n_0$. However, it is possible to give upper and lower bounds, which are still useful. One upper bound was given in \cite{lowerBoundRn}. This was successively improved in \cite{trajectoryLength}, \cite{worseUpperBoundRn} and \cite{upperBoundRn}. To the best of our knowledge, \cite{upperBoundRn} gave the best estimate so far, which is
\begin{align*}
    \mathcal{R}(\mathbf{n}) \leq \sum_{(j_1,\dots,j_L) \in J} \prod_{l=1}^L {n_l \choose j_l},
\end{align*}
where
\begin{align*}
    J = \{(j_1,\dots,j_L) \in \mathbb{Z}^L : 0 \leq j_l \leq \min\{n_0, n_1-j_1, \dots, n_{l-1}-j_{l-1}, n_l\}\}.
\end{align*}
We will give an even sharper version of this bound.

\begin{theorem}[sharper version of \cite{upperBoundRn}, Theorem 1]
\label{theorem: upper bound R(n)}
For a general neural network we have
\begin{equation*}
    \mathcal{R}(\mathbf{n}) \leq \sum_{(j_1,\dots,j_L) \in J} \prod_{l=1}^L f_{j_l,d_l}(n_l),
\end{equation*}
where
\begin{align*}
    d_l &= \min\{n_0,n_1-j_1,\dots,n_{l-1}-j_{l-1}\} \quad \text{for $l=1,2,\dots,L$,} \\
    J &= \{(j_1,\dots,j_L) \in \mathbb{Z}^L : 0 \leq 2j_l \leq n_l + \min\{n_l,d_l\}, j_L \leq \min\{n_L,d_L\}\}, \\
    f_{j,d}(n) &=
    \begin{cases}
    {n \choose j} &\text{for $j < d$,} \\
    {n-2j+2d-1 \choose d-1} + {n-2j+2d-2 \choose d-1} &\text{for $j \geq d$.}
    \end{cases}
\end{align*}
\end{theorem}

It is not obvious that this is indeed better than the previous bound. A proof can be found in the end of \Cref{Section 2}. As already mentioned, it is also possible to give a lower bound for the expressive power of a neural network by constructing specific coefficients for the affine functions that lead to a lot of linear regions. Different lower bounds were given in \cite{lowerBoundRn}, \cite{differentLowerBoundRn} and \cite{upperBoundRn}. \\

Of course, the expressive power $\mathcal{R}(\mathbf{n})$ does not tell the whole story about what functions are contained in $\mathcal{F}_{\mathbf{n}}$. Not every function with $\mathcal{R}(\mathbf{n})$ linear regions can be calculated by a neural network with this design. A counterexample is the function
\begin{align*}
    h(x) = 1-\max\{0,1-3x\}-\max\{0,3x-1\}+\max\{0,6x-4\}.
\end{align*}
As shown in \Cref{Section 3} it cannot be calculated by any network of width $2$. But by \Cref{theorem: R(n) in 1 dimension} such a network can have enough linear regions. While not every function with $\mathcal{R}(\mathbf{n})$ linear regions can be calculated, we will prove later that there is a small enough $k \leq \mathcal{R}(\mathbf{n})$ such that this holds. Therefore, we propose the definition of another expressive power $\mathcal{E}(\mathbf{n})$. \\

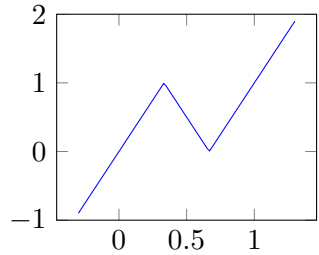
\begin{figure}[h]
    \centering
    \begin{tikzpicture}
    \begin{axis}[no markers, ymin=-1, ymax=2]
    \addplot+[domain=-0.3:1.3,samples=100,color=blue]{1-max(1-3*x,0)-max(3*x-1,0)+max(6*x-4,0)};
    \end{axis}
    \end{tikzpicture}
    \caption{Graph of $h$}
\end{figure}

\begin{definition}
Let $\mathcal{E}(\mathbf{n})$ be the maximal number $k$ such that any piecewise affine function with at most $k$ linear regions is in $\mathcal{F}_{\mathbf{n}}$.
\end{definition}

As with the other expressive power we only consider the case $n_{L+1} = 1$. One can easily conclude lower bounds for the general case from this by just using $n_{L+1}$ copies of the network with $n_{L+1} = 1$. Again, we start by giving some results in the case of $1$-dimensional input. Our first result deals with networks consisting of layers that have width $3$.

\begin{theorem}
\label{theorem: E(n) in 1 dimension with width 3}
For the network consisting of $L$ layers of $3$ neurons and $n_0 = 1$ we have $\mathcal{E}(\mathbf{n}) \geq L+2$.
\end{theorem}

For wider networks we can obtain a different result, which is easier to prove. Roughly speaking, this is because wider networks have enough space to store intermediate results and perform calculations in different neurons.

\begin{theorem}
\label{theorem: E(n) in 1 dimension}
For any network with $n_0 = 1$ and $n_l \geq 4$ we have
\begin{equation*}
    \mathcal{E}(\mathbf{n}) \geq 2 + \sum_{l=1}^L (n_l-4).
\end{equation*}
\end{theorem}

Naturally, we would like to have similar results in the case of arbitrary input dimension. One step in this direction is given by the following theorem.

\begin{theorem}
\label{theorem: lower bound E(n)}
Any compactly supported piecewise affine function with $k$ linear regions can be calculated by a neural network with $O(k^{(n_0+1)^2})$ layers of width $2n_0+6$.
\end{theorem}

A few easy yet very interesting consequences of this result are universal approximation theorems. These theorems ensure that neural networks can approximate any function of a certain class at least in theory. Usually, they are proved in a different way, for example by approximating polynomials with a register model \cite{universalApproximation}. Instead of taking the detour over polynomials we can directly use piecewise affine functions.

\begin{corollary}[follows from \cite{universalApproximation}, Theorem 3.2 and Theorem 4.16]
\label{corollary: lower bound E(n) universal approximation}
ReLU networks with width $2n_0+6$ and arbitrary depth can approximate any function in $L^p(\mathbb{R}^{n_0})$ with $p \in [1,\infty)$ or $C(K)$, where $K \subset \mathbb{R}^{n_0}$ is compact.
\end{corollary}

The quantitative nature of \Cref{theorem: lower bound E(n)} also allows us to explicitly state the error in terms of the width and depth of the network. We will call this the approximation rate. Optimal approximation rates for the $L^p$-norms with $p \in [1,\infty)$ are already known \cite{optimalApproximationRate}. For any Lipschitz continuous function $f$ there are a ReLU-network of depth $L$ and width $N$ calculating a function $g$ such that
\begin{align*}
    \|f-g\|_{L^p([0,1]^{n_0})} = O(\sqrt{n_0} \Lip(f) (N^2 L^2 \ln(N))^{-\frac{1}{n_0}}).
\end{align*}
But the technique used in \cite{optimalApproximationRate} does not easily carry over to the $W^{1,1}$-norm because they used piecewise constant functions for this approximation. Nevertheless, it is possible to get a similar result for the Sobolev norm using the expressive power $\mathcal{E}(\mathbf{n})$.

\begin{theorem}
\label{theorem: approximation rate in sobolev norm}
For any $f \in W^{1,1}_0([0,1]^{n_0})$ with Lipschitz continuous derivative there is a ReLU network with width $2n_0+6$ and depth $L$ calculating a function $g$ such that
\begin{equation*}
    \|f-g\|_{W^{1,1}} \leq C_{n_0} \Lip(Df) L^{-\frac{1}{n_0}}.
\end{equation*}
\end{theorem}

Note that we do not know if this is the optimal approximation rate. Furthermore, one can even prove a more general version of \Cref{theorem: approximation rate in sobolev norm}, which takes wider networks into account. This would require a generalized version of \Cref{theorem: lower bound E(n)}. In the end one could obtain an approximation rate of 
\begin{align*}
    C_{n_0} \Lip(Df) (NL)^{-\frac{1}{n_0}},
\end{align*}
where $N$ is the width of the network. We chose not to do that because the proof of \Cref{theorem: lower bound E(n)} would become more tedious. \\

At this point we have two different expressive powers that have quite compelling consequences on their own. The crucial observation is that they are somewhat dual. This is the reason why it is genuinely promising to combine them. One striking consequence of this combination are width vs depth results. For example one might ask the simple question whether there is a deep network that cannot be approximated by any shallow network unless it has exponential width. This was answered in \cite{widthInefficiency}. They showed that there is a function $g$ calculated by a ReLU-network with $2k^3+8$ layers, $3k^3+12$ neurons in total and coefficients $\leq 4$ such that
\begin{align*}
    \|f-g\|_{L^1([0,1]^{n_0})} \geq \frac{1}{64}
\end{align*}
for any function $f$ calculated by networks with $k$ layers and $\leq 2^k$ neurons in total. Our pair of expressive powers can provide an improvement of this result in dimension $1$.

\begin{theorem}
\label{theorem: width inefficiency in 1 dimension}
Let $L \geq 2$. There is a neural network with $n_0 = 1$, depth $L^2$, width $3$ and coefficients $\leq 6$ that can compute a function $g$ with the following property: For any function $f$ calculated by a neural network of depth $L$ and width $< 3^{L-1} - 1$ we have
\begin{align*}
    \|f-g\|_{L^1([0,1])} \geq \frac{1}{9}.
\end{align*}
\end{theorem}

Arguably, the opposite scenario is even more fascinating: Are there any shallow neural networks that cannot be approximated by a deep neural network unless it has exponential depth? This question was raised in \cite{widthDepthQuestion} and remained unresolved. The notion of dual expressive powers is strong enough to give an answer.

\begin{theorem}
\label{theorem: depth efficiency}
Any compactly supported function calculated by a neural network with $L$ hidden layers of size $N$ can also be calculated by a neural network with $O(N^{Ln_0(n_0+1)^2})$ hidden layers of size $2n_0+6$.
\end{theorem}

The upshot is that deep narrow ReLU networks can approximate wide shallow ReLU networks reasonably well but this does not work the other way around. \\

The only slightly annoying part of this conclusion is that we only worked with ReLU as our activation function up until now. But this is not a major problem since we can export these results to a whole class of other activation functions. The idea is to replace ReLU by an activation function that can compute ReLU.

\begin{definition}
We call a function $\rho \in C(\mathbb{R})$ \textit{ReLU-computing} if it has a point with nonzero derivative and a point $\alpha \in \mathbb{R}$ such that $\rho$ is $C^3$ in a neighborhood of $\alpha$ and $\rho''(\alpha) \neq 0$.
\end{definition}

In order to export results, we need to be able to simulate a ReLU-network with the help of a ReLU-computing activation function. It suffices to construct a network with a ReLU-computing activation function that estimates ReLU itself. Then we can just replace each instance of ReLU in the original network by this ReLU-computing network.

\begin{theorem}
\label{theorem: ReLU-computing calculates relu}
Let $\varepsilon > 0$. A neural network with $L$ layers of size $N$, coefficients bounded by $C$ and ReLU activation function can be approximated on $[-1,1]$ by a neural network with $O(L\ln(\varepsilon^{-1}L)^3 + L^4\ln(CN)^3)$ layers of size $8N$ with a ReLU-computing activation function $\rho$ up to an error of $\varepsilon$.
\end{theorem}

As with the expressive powers, we will also use a somewhat dual class of activation functions, which we call ReLU-computable.

\begin{definition}
\label{def: ReLU-computable}
A function $\rho: \mathbb{R} \to \mathbb{R}$ is called \textit{ReLU-computable} if it satisfies the following properties:
\begin{enumerate}
    \item There is a piecewise affine function $l: \mathbb{R} \to \mathbb{R}$ such that
    \begin{align*}
        \rho(x) - l(x) \xrightarrow{|x| \to \infty} 0
    \end{align*}
    exponentially quickly.
    \item $\rho$ is smooth, $|\rho(0)| \leq 1$
    \begin{align*}
        \frac{1}{n!} \left|\frac{\mathrm{d}^n}{\mathrm{d}x^n}\rho(x)\right| &\leq 1 \quad \text{for all $n \geq 1$ and} \\ 
        \frac{1}{n!} \frac{\mathrm{d}^n}{\mathrm{d}x^n} \rho(x) &\xrightarrow{n \to \infty} 0
    \end{align*}
    exponentially quickly uniformly in $x$.
\end{enumerate}
\end{definition}

It is worth noting that a lot of commonly used activation functions like Softplus, Gaussian or Sigmoid are ReLU-computing as well as ReLU-computable. We will show this in \Cref{Section 6}. \\

Similar to above, we need to be able to simulate a network with a ReLU-computable activation function using a ReLU-network. This can be done with the exact same idea. We explicitly construct a ReLU-network that approximates the ReLU-computable activation function. This allows us to just replace the ReLU-computable activation function in each neuron of the original network by our newly constructed ReLU-network.

\begin{theorem}
\label{theorem: relu calculates ReLU-computable}
Let $\varepsilon > 0$. A neural network with $L$ layers of size $N$, coefficients bounded by $C$ and a ReLU-computable activation function $\rho$ can be approximated on $[-1,1]$ by a neural network with $O(L\ln(\varepsilon^{-1}L)^3 + L^4\ln(CN)^3)$ layers of size $11N$ with ReLU activation function up to an error of $\varepsilon$.
\end{theorem}

Ultimately, we can use the dual results \Cref{theorem: ReLU-computing calculates relu} and \Cref{theorem: relu calculates ReLU-computable} to export a whole range of results for ReLU networks to other activation functions. An easy example of this process is the width inefficiency result \Cref{theorem: width inefficiency in 1 dimension}.

\begin{theorem}
\label{theorem: width inefficiency for ReLU-computable activation}
A neural network with ReLU-computing activation function $\rho$ and $O(L^{40})$ hidden layers of size $24$ can compute a function that cannot be approximated in $L^1([0,1])$ by any neural network with a ReLU-computable activation function of depth $L$ unless it has exponential width or coefficients.
\end{theorem}

To the best of our knowledge, results of this type were not known yet. The same process could also be applied to \Cref{theorem: depth efficiency}, \Cref{theorem: approximation rate in sobolev norm}, universal approximation theorems or any other result of that nature. The only obstacle is the size of the coefficients $C$ in \Cref{theorem: ReLU-computing calculates relu} and \Cref{theorem: relu calculates ReLU-computable}. \\

To sum up, any result for ReLU networks also holds for other reasonable activation functions as long as one can control the size of the coefficients. In particular, wide shallow neural networks cannot always approximate deep narrow neural networks reasonably well.

\newpage

\section{Estimates for $\mathcal{R}(\mathbf{n})$}
\label{Section 2}

\subsection{$1$-dimensional input}

Before we start with calculating the expressive power we need to do some ground work. The notion of neurons can be extended to enhanced neurons.

\begin{definition}[\cite{universalApproximation}, Section 4]
Let $k,m,n \in \mathbb{N}$. A \textit{hidden layer of $m$ enhanced neurons} is defined as $g \circ \rho \circ f$, where $f: \mathbb{R}^k \to \mathbb{R}^m$ and $g: \mathbb{R}^m \to \mathbb{R}^n$ are affine and $\rho$ is the activation function.
\end{definition}

Note that the additional affine function $g$ can be absorbed into the affine function of the next layer. Therefore, using enhanced neurons will result in the same set of functions as using ordinary neurons. The key difference is that it is considerably easier to work with them. \\

Another worthwhile simplification is working with $n_{L+1} = 1$. This does not change the expressive power either, which we will show in the following.

\begin{proof}[Proof of \Cref{theorem: n_L+1=1}]
Pick some weights for the neural network such that the number of linear regions is equal to $\mathcal{R}(\mathbf{n})$ for arbitrary $n_{L+1}$. Let $A^{(k)}$ be the matrices corresponding to the derivatives in the linear regions. Pick $z \in \mathbb{R}^{n_{L+1}}$ with linearly independent entries over $\mathbb{Q}[A^{(k)}_{ij}]$. This is possible because $\mathbb{Q}[A^{(k)}_{ij}]$ is countable but $\mathbb{R}$ is uncountable. Now, consider the neural network with output dimension $1$ that calculates the scalar product of $z$ and the output of the other neural network. \\

Assume that it has less linear regions. Then there are $k$ and $l$ such that the linear regions corresponding to $A^{(k)} \neq A^{(l)}$ merge together. This would imply $z^T A^{(k)} = z^T A^{(l)}$, which contradicts linear independence. Hence, the new neural network with output dimension $1$ has the same number of linear regions as the old neural network, so we can assume $n_{L+1} = 1$.
\end{proof}

With the help of this theorem we can start evaluating the usual expressive power $\mathcal{R}(\mathbf{n})$ for $1$-dimensional input. $\mathcal{R}(\mathbf{n})$ counts the number of linear regions. In order to calculate it, we also need the number of non-constant linear regions.

\begin{definition}
Let $\Tilde{\mathcal{R}}(\mathbf{n})$ be the maximal number of non-constant linear regions of a function in $\mathcal{F}_{\mathbf{n}}$. Similarly, we write $\mathcal{R}(f)$ and $\Tilde{\mathcal{R}}(f)$ for the number of linear regions and non-constant linear regions of the function $f$, respectively.
\end{definition}

\begin{definition}
A point in the boundary of a linear region is called \textit{breakpoint}.
\end{definition}

Non-constant linear regions are important because additional layers can break up those linear regions into even more linear regions. In contrast to that, constant regions will always stay constant and therefore do not increase the value of $\mathcal{R}(\mathbf{n})$ any further. The details of this phenomenon can be seen in the following proof.

\begin{proof}[Proof of \Cref{theorem: R(n) in 1 dimension}]
We prove the lower bound by explicitly constructing a neural network with the desired number of linear regions. A hidden layer with $n$ neurons will compute the function $g_n$, which we will define in a moment. The graphs of all of those functions will have an interesting behavior in $[0,1]^2$ but this will be made more explicit later on. Let's start with the case $n \geq 3$, where we use the following neurons:
\begin{align*}
    -\frac{3}{2}&\max\{0,(2n+1)x-1\} \\
    &\max\{0,(2n+1)x-3\} \\
    -&\max\{0,-(2n+1)x+5\} \\
    (-1)^i&\max\{0,(2n+1)x-(2i-1)\} \quad \text{for } i=4,\dots,n
\end{align*}
Let $g_n$ be the sum of those neurons plus $5$. Obviously, $g_n$ is a piecewise affine function with breakpoints at $x=\frac{2j-1}{2n+1}$ for $j=1,\dots,n$. For $j=4,\dots,n$ we get
\begin{align*}
    g_n\left(\frac{2j-1}{2n+1}\right) &= 5 - \frac{3}{2}(2j-2) + (2j-4) + \sum_{i=4}^j (-1)^i ((2j-1)-(2i-1)) \\
    &= 
    \begin{cases}
        0 &\text{if $j$ is even}, \\
        1 &\text{if $j$ is odd},
    \end{cases}
\end{align*}
where the last equality can be shown via a simple induction on $j$. For $j=1,2,3$ a direct calculation shows that the pattern persists. Furthermore, we have $g(0) = 1$ and $g(1) = (-1)^n$. \\

\begin{figure}[h]
    \centering
    \begin{tikzpicture}
    \begin{axis}[no markers, ymin=-1, ymax=2]
    \addplot+[domain=-0.3:1.3,samples=100,color=blue]{5-3/2*max(0,9*x-1)+max(0,9*x-3)-max(0,-9*x+5)+max(0,9*x-7)};
    \addplot+[dashed,color=black] coordinates {(0,0)(0,1)};
    \addplot+[dashed,color=black] coordinates {(0,1)(1,1)};
    \addplot+[dashed,color=black] coordinates {(1,1)(1,0)};
    \addplot+[dashed,color=black] coordinates {(1,0)(0,0)};
    \end{axis}
    \end{tikzpicture}
    \caption{Graph of $g_4$}
\end{figure}
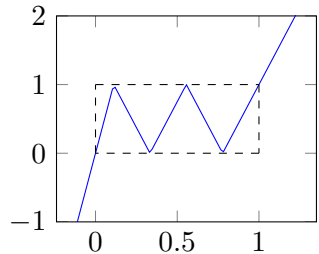

For $n=2$ we define
\begin{equation*}
    g_2(x) = 1 - \max\{0,-3x+1\} - \max\{0,3x-2\},
\end{equation*}
which is shown in \Cref{graph of g2}. \\

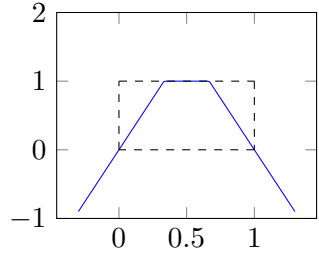
\begin{figure}[h]
    \centering
    \begin{tikzpicture}
    \begin{axis}[no markers, ymin=-1, ymax=2]
    \addplot+[domain=-0.3:1.3,samples=100,color=blue]{1-max(0,-3*x+1)-max(0,3*x-2)};
    \addplot+[dashed,color=black] coordinates {(0,0)(0,1)};
    \addplot+[dashed,color=black] coordinates {(0,1)(1,1)};
    \addplot+[dashed,color=black] coordinates {(1,1)(1,0)};
    \addplot+[dashed,color=black] coordinates {(1,0)(0,0)};
    \end{axis}
    \end{tikzpicture}
    \caption{Graph of $g_2$}
    \label{graph of g2}
\end{figure}

Now, we claim that the neural network that computes
\begin{align*}
    g = g_{n_l} \circ g_{n_{l-1}} \circ \dots \circ g_{n_1}
\end{align*}
has the right number of linear regions. The key observation is the following: Any time $g_{n_1}$ runs through the values from $0$ to $1$ the function $g_{n_2} \circ g_{n_1}$ will traverse the graph of $g_{n_2}$ from $0$ to $1$. \\

\begin{figure}[h]
    \centering
    \begin{tikzpicture}
    \begin{axis}[no markers, ymin=-1, ymax=2]
    \addplot+[domain=-0.3:1.3,samples=100,color=blue]{5-3/2*max(0,7*x-1)+max(0,7*x-3)-max(0,-7*x+5)};
    \addplot+[dashed,color=black] coordinates {(0,0)(0,1)};
    \addplot+[dashed,color=black] coordinates {(0,1)(1,1)};
    \addplot+[dashed,color=black] coordinates {(1,1)(1,0)};
    \addplot+[dashed,color=black] coordinates {(1,0)(0,0)};
    \end{axis}
    \end{tikzpicture}
    \begin{tikzpicture}[
        evaluate={
            function f(\x) {
                return 5-3/2*max(0,7*\x-1)+max(0,7*\x-3)-max(0,-7*\x+5);
            };
        },
    ]
    \begin{axis}[no markers,ymin=-1,ymax=2]
    \addplot+[domain=-0.3:1.3,samples=200,color=blue]{1-max(0,-3*f(x)+1)-max(0,3*f(x)-2)};
    \addplot+[dashed,color=black] coordinates {(0,0)(0,1)};
    \addplot+[dashed,color=black] coordinates {(0,1)(1,1)};
    \addplot+[dashed,color=black] coordinates {(1,1)(1,0)};
    \addplot+[dashed,color=black] coordinates {(1,0)(0,0)};
    \end{axis}
    \end{tikzpicture}
    \caption{Graphs of $g_3$ and $g_2 \circ g_3$}
\end{figure}
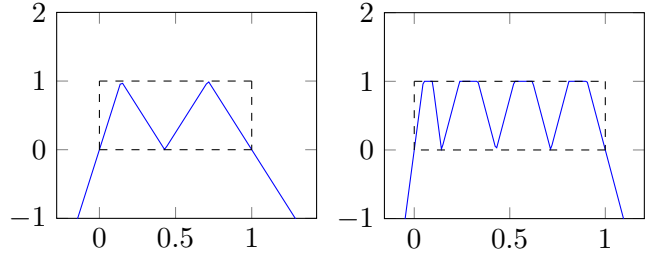

Hence, any non-constant linear region of $g_{n_1}$ in $[0,1]$ gives rise to a scaled copy of the graph of $g_{n_2}$ in the dashed box, which contains all of the interesting behavior of $g_{n_2}$. The constant linear regions of $g_{n_1}$ will just carry over to $g_{n_2} \circ g_{n_1}$. In particular, for $n_2 \geq 2$ we get
\begin{align*}
    \mathcal{R}(g_{n_2} \circ g_{n_1}) 
    &= \mathcal{R}(g_{n_2}) \Tilde{\mathcal{R}}(g_{n_1}) + \mathcal{R}(g_{n_1}) - \Tilde{\mathcal{R}}(g_{n_1}) \\
    &= (\mathcal{R}(g_{n_2})-1) \Tilde{\mathcal{R}}(g_{n_1}) + \mathcal{R}(g_{n_1}) \quad \text{and} \\
    \Tilde{\mathcal{R}}(g_{n_2} \circ g_{n_1}) &= \Tilde{\mathcal{R}}(g_{n_2}) \Tilde{\mathcal{R}}(g_{n_1}).
\end{align*}
By repeatedly applying this observation we obtain
\begin{align*}
    \mathcal{R}(g) &= 1 + \sum_{i=1}^L (\mathcal{R}(g_{n_i})-1) \prod_{j=1}^{i-1} \Tilde{\mathcal{R}}(g_{n_j}) \\
    &= 1 + \sum_{i=1}^L n_i \prod_{j=1}^{i-1} (n_j + \mathbbm{1}_{n_j > 2}),
\end{align*}
which proves the lower bound for $\mathcal{R}(\mathbf{n})$. \\

For the upper bound we group together adjacent linear regions with the same sign of the slope and call them monotone regions. We use the notation $\mathcal{R}^m$ for the monotone regions. First notice that $\mathcal{R}^m(n_i) \leq \mathcal{R}(n_i) \leq n_i + 1$ for $n_i > 2$ since $n_i$ neurons in one layer can produce at most $n_i$ breakpoints. \\

Furthermore, we have $\mathcal{R}(2) \leq 3$ for the same reason and $\mathcal{R}^m(2) \leq 2$, which can be proven by a simple case distinction: If we have a constant linear region or only two linear regions, the claim is clear. Otherwise the slope of the linear region in the middle is the sum of the slopes of the linear regions on both sides. Hence, the middle region will be in the same monotone region as the linear region with the bigger absolute value of the slope. \\

If we add a new layer with $n_i$ neurons, each monotone region is split up into at most $\mathcal{R}^m(n_i)$ monotone regions. Hence, we have
\begin{align*}
    \mathcal{R}^m(n_1,n_2,\dots,n_i) &\leq \prod_{j=1}^i \mathcal{R}^m(n_j) \\
    &\leq \prod_{j=1}^i (n_j + \mathbbm{1}_{n_j > 2}).
\end{align*}
For linear regions we can keep all of the breakpoints from the previous layer. Additionally, we can add at most $\mathcal{R}(n_i)-1$ breakpoints to each monotone region. This leads to the same recursion as in the lower bound and eventually proves the upper bound as well.
\end{proof}

Before moving on to the case of general input dimension, we take some time to ponder about the consequences of \Cref{theorem: R(n) in 1 dimension}. Layers with $2$ neurons play a special role because of the indicator function $\mathbbm{1}_{n_j > 2}$ in the formula. They form a bottleneck, which we will prove below.

\begin{proof}[Proof of \Cref{corollary: R(n) in 1 dimension bottleneck}]
Assume that $n_l = 2$ for some $1 \leq l \leq L$. Define $\Tilde{n}_i$ by swapping $n_l$ to the end, i.e.
\begin{align*}
    \Tilde{\mathbf{n}} = (1,n_1,\dots,n_{l-1},n_{l+1},\dots,n_L,n_l,1).
\end{align*}
Observe that for every $i = l,\dots,L$ we have
\begin{align*}
    \frac{n_i \prod_{j=1}^{i-1} (n_j + \mathbbm{1}_{n_j > 2})}{\Tilde{n}_i \prod_{j=1}^{i-1} (\Tilde{n}_j + \mathbbm{1}_{\Tilde{n}_j > 2})} &= \frac{n_i (n_l + \mathbbm{1}_{n_l > 2})}{n_{i+1} (n_i + \mathbbm{1}_{n_i > 2})} \leq 1 \\
    \Rightarrow n_i \prod_{j=1}^{i-1} (n_j + \mathbbm{1}_{n_j > 2}) &\leq \Tilde{n}_i \prod_{j=1}^{i-1} (\Tilde{n}_j + \mathbbm{1}_{\Tilde{n}_j > 2})
\end{align*}
For $i = 1,\dots,l-1$ we even have equality. Summing over $i = 1,\dots,L$ gives
\begin{align*}
    \mathcal{R}(\mathbf{n}) \leq \mathcal{R}(\Tilde{\mathbf{n}}),
\end{align*}
which shows that swapping a hidden layer with $2$ neurons to the end can only increase the expressive power. This was the first part of \Cref{corollary: R(n) in 1 dimension bottleneck}. \\

For the second part we consider the case that $n_i \geq 3$ for $i = 1,\dots,l-1$. Then we have
\begin{align*}
    \mathcal{R}(\mathbf{n}) &= 1 + \sum_{i=1}^L n_i \prod_{j=1}^{i-1} (n_j + \mathbbm{1}_{n_j > 2}) \\
    &= 1 + \sum_{i=1}^{l-1} n_i \prod_{j=1}^{i-1} (n_j + \mathbbm{1}_{n_j > 2}) + \sum_{i=l}^L n_i \prod_{j=1}^{i-1} (n_j + \mathbbm{1}_{n_j > 2}) \\
    &= \prod_{i=1}^{l} (n_i + 1) + \sum_{i=l}^L n_i \prod_{j=1}^{i-1} (n_j + \mathbbm{1}_{n_j > 2}).
\end{align*}
Hence, the first $l-1$ layers can be swapped arbitrarily without effecting the expressive power.
\end{proof}

The second consequence of \Cref{theorem: R(n) in 1 dimension} is that we can find the optimal network architecture with respect to this expressive power.

\begin{proof}[Proof of \Cref{corollary: R(n) in 1 dimension optimal network architecture}]
First of all, we can swap all layers with $2$ neurons to the end by \Cref{corollary: R(n) in 1 dimension bottleneck}. There is no point in having a layer with $n \geq 6$ neurons because a layer with $n-3$ and a layer with $3$ neurons will produce more linear regions due to
\begin{align*}
    (3+1)((n-3)+1) \geq 6+1.
\end{align*}
Similarly, a layer with $5$ neurons should be split into a layer with $2$ and a layer with $3$ neurons. Layers with $4$ neurons can be replaced by two layers with $2$ neurons each. \\

Hence, an "optimal" network consists of $L-k$ layers of $3$ neurons followed by $k$ layers of $2$ neurons. Plugging this into \Cref{theorem: R(n) in 1 dimension} we get that the maximal number of linear regions of such a network is equal to
\begin{align*}
    1 + \sum_{i=1}^{L-k} 3 \cdot 4^{i-1} + \sum_{i=1}^{k} 2^i \cdot 4^{L-k} &= 1 + (4^{L-k} - 1) + (2^{k+1}-2) 4^{L-k} \\
    &= (2^{k+1}-1) 4^{L-k}
\end{align*}.
This is monotonically decreasing in $k$, so $k$ needs to be as small as possible.
\end{proof}

\newpage

\subsection{$n_0$-dimensional input}

In this section we turn our focus to arbitrary inputs instead of just $1$-dimensional input. If the neural network has just $1$ hidden layer, then the maximal number of linear regions is given by Zaslavsky's theorem as noted in \cite{lowerBoundRn}.

\begin{theorem}[Zaslavsky's theorem, \cite{zaslavskyTheorem}, Proposition 2.4]
\label{zaslavskyTheorem}
An arrangement of $m$ hyperplanes can cut $\mathbb{R}^d$ into at most $\sum_{i=0}^d {m \choose i}$ regions.
\end{theorem}

The application of \Cref{zaslavskyTheorem} in the case of $1$ hidden layer can be explained quite easily: The input of each neuron is an affine function. The neuron switches from inactive to active when the input crosses $0$, which defines a hyperplane. Since the output is a linear combination of those neurons, the linear regions will be defined by those hyperplanes. This provides an exact formula for $\mathcal{R}(n_0,n_1,1)$.

\begin{corollary}[\cite{lowerBoundRn}, Proposition 1]
We have
\begin{equation*}
    \mathcal{R}(n_0,n_1,1) = \sum_{i=0}^{\min\{n_0,n_1\}} {n_1 \choose i}.
\end{equation*}
\end{corollary}

\Cref{zaslavskyTheorem} is not only important in the case of $1$ hidden layer. We can also use it to examine the effects of $1$ additional hidden layer in a deeper network. This additional layer will split up existing linear regions into more linear regions, which can be estimated with the following lemma.

\begin{lemma}[\cite{upperBoundRn}, Lemma 4]
\label{subspace Zaslavsky}
For $W \in \mathbb{R}^{d \times m}$ and $b \in \mathbb{R}^d$ we consider the equation $Wx + b = 0$, which defines $m$ hyperplanes in $\mathbb{R}^d$. Those hyperplanes divide $\mathbb{R}^d$ into at most $\sum_{i=0}^{\rank(W)} {m \choose i}$ regions.
\end{lemma}

\Cref{subspace Zaslavsky} emphasizes that the rank of a linear region plays a crucial role. A higher rank will cause a higher number of linear regions later on. Therefore, we would like to estimate how many linear regions of a certain rank are possible.

\begin{definition}[\cite{upperBoundRn}, Section 2]
We say that a neuron is \textit{active} if its input is $>0$. The \textit{activation set} is the set of active neurons.
\end{definition}

The rank of a linear region is bounded by the size of the activation set of the corresponding layer. Thus, we will try to estimate that quantity instead.

\begin{lemma}
\label{activation patterns 1 layer}
Consider a neural network with $d$ input neurons and $1$ hidden layer with $n$ neurons. The number of linear regions with activation set of size $\geq n-m$ is bounded by
\begin{equation*}
    \sum_{j=0}^m f_{j,d}(n),
\end{equation*}
where
\begin{equation*}
    f_{j,d}(n) =
    \begin{cases}
    {n \choose j} &\text{for $j < d$,} \\
    {n-2j+2d-1 \choose d-1} + {n-2j+2d-2 \choose d-1} &\text{for $j \geq d$.}
    \end{cases}
\end{equation*}
In particular, we have $f_{j,d}(n) = 0$ for $j > \frac{n+\min\{d,n\}}{2}$.
\end{lemma}

\begin{proof}
We use induction on $d$. \\

\textit{Base Case:} In the case $d=1$ the problem is equivalent to having $n$ points on the real line, which all have an active and an inactive side. Obviously, it would be pointless if two neurons had the same breakpoint, so we can assume that the points are distinct. The $n$ points split the real line into $n+1$ regions and for each region we count the number of active points. \\

If there are two points whose active sides do not intersect, we can just swap them. This can not decrease the counter of any region. Hence, we can assume that we are in the following situation. On the left hand side of the real line we have $k$ points that have there active sides on their right. And on the right hand side there are $n-k$ points that have their active sides on their left. \\

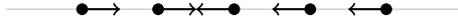
\begin{figure}[h]
    \centering
    \begin{tikzpicture}
    \draw[gray!50] (0,0) -- (6,0);
    \filldraw[black] (1,0) circle (2pt);
    \draw[black, thick, ->] (1,0) -- (1.5,0);
    \filldraw[black] (2,0) circle (2pt);
    \draw[black, thick, ->] (2,0) -- (2.5,0);
    \filldraw[black] (3,0) circle (2pt);
    \draw[black, thick, ->] (3,0) -- (2.5,0);
    \filldraw[black] (4,0) circle (2pt);
    \draw[black, thick, ->] (4,0) -- (3.5,0);
    \filldraw[black] (5,0) circle (2pt);
    \draw[black, thick, ->] (5,0) -- (4.5,0);
    \end{tikzpicture}
    \caption{"Optimal" configuration for $d=1$ and $n=5$}
\end{figure}

All of the points are active in the linear region after the $k$th point. The number of active points decreases by $1$ every time you go one region to the right (and the same holds true for left). Thus, it is optimal to have the same number of points facing right and facing left, which gives exactly the claimed upper bound. \\

\textit{Induction Step:} Let's assume we have already proven our result for $d$ input neurons. Now, we use induction on $n$ to prove it for $d+1$. For $n < d$ the result is clear because there are only
\begin{equation*}
    f_{j,d}(n) = {n \choose j} = {n \choose n-j}
\end{equation*}
ways to choose the $n-j$ neurons for the activation set. In the case $n = d$ it is easy to prove that we still have $f_{j,d}(n) = {n \choose j}$ for all $j$, so the result is true. This will serve as our base case for the induction on $n$. \\

For the induction step we need the following observations. All regions are convex because they are intersections of convex regions. If a new hyperplane intersects an existing region with activation set of size $k$, it will produce two regions with activation sets of size $k$ and $k+1$, respectively. We identify the region with activation set of size $k+1$ with the old region and the other one with the $d$-dimensional region on the new hyperplane. The activation sets of the untouched regions either stay the same or increase by $1$ if they are on the right side of the new hyperplane. \\

\begin{figure}[h]
    \centering
    \begin{tikzpicture}[scale = 0.6]
    \draw[gray!50] (1,6) -- (3,0);
    \draw[gray!50] (4,6) -- (2,0);
    \draw[gray!50] (0,5) -- (6,3);
    \draw[black] (0,2) -- (6,4);
    \draw[black, ->] (3,3) -- (2.8,3.6);
    \node at (0.5,3.5) {+1};
    \node at (2.3,3.5) {+1};
    \node at (3.5,3.5) {+1};
    \node at (2.3,5.3) {+1};
    \node at (0.5,5.3) {+1};
    \node at (4.5,4.8) {+1};
    \end{tikzpicture}
    \caption{Adding a new hyperplane}
\end{figure}
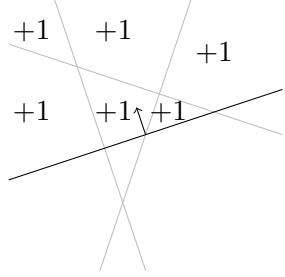

Hence, for every old region with activation set of size $k$ there is at most $1$ region with activation set of size $k+1$. Additionally, we can count the other regions using the $d$-dimensional result on the new hyperplane. Thus, the number of linear regions with activation set of size $\geq (n+1) - m$ is bounded by
\begin{align*}
    &\sum_{j=0}^m f_{j,d+1}(n) + \sum_{j=0}^{m-1} f_{j,d}(n) \\
    &= \sum_{j=0}^m f_{j,d+1}(n) + \sum_{j=1}^m f_{j-1,d}(n) \\
    &= \sum_{j=0}^d {n \choose j} + \sum_{j=d+1}^m {n-2j+2d+1 \choose d} + {n-2j+2d \choose d} \\
    &+ \sum_{j=1}^d {n \choose j-1} + \sum_{j=d+1}^m {n-2j+2d+1 \choose d-1} + {n-2j+2d \choose d-1} \\
    &= \sum_{j=0}^d {n+1 \choose j} + \sum_{j=d+1}^m {n-2j+2d+2 \choose d} + {n-2j+2d+1 \choose d} \\
    &= \sum_{j=0}^m f_{j,d+1}(n+1),
\end{align*}
where we used Pascal's rule in the penultimate step.
\end{proof}

With this in mind we have all the necessary ingredients to prove a general upper bound for $\mathcal{R}(\mathbf{n})$. The strategy is to estimate the effect of each additional layer using \Cref{subspace Zaslavsky} and \Cref{activation patterns 1 layer}.

\begin{proof}[Proof of \Cref{theorem: upper bound R(n)}]
Each additional layer partitions the linear regions from the previous layer. If the new layer has $n_l$ neurons, a region from the previous layer might be cut into smaller regions by up to $n_l$ hyperplanes. Hence, we try to recursivley bound the number of smaller regions within another region. Let $R(l,d)$ be the maximal number subregions obtainable with the layers $l,l+1,\dots,L$ from a region with $d$-dimensional image. By \Cref{subspace Zaslavsky} we have
\begin{equation*}
    R(L,d) = \sum_{j=0}^{\min\{n_L,d\}} {n_L \choose j} = \sum_{j=0}^{\min\{n_L,d\}} f_{j,d}(n_L).
\end{equation*}
Furthermore, we have
\begin{equation*}
    R(l,d) \leq \sum_{j=0}^{n_l} N_{j,d,n_l} R(l+1,\min\{j,d\}),
\end{equation*}
where $N_{j,d,n_l}$ is the number of subregions in the next layer with activation set of size $j$ in the optimal partition that leads to the highest expressive power. Since $R(l+1,\min\{j,d\})$ is increasing in $j$, we get
\begin{equation*}
    R(l,d) \leq \sum_{j=0}^{n_l} f_{j,d}(n_l) R(l+1,\min\{j,d\})
\end{equation*}
using \Cref{activation patterns 1 layer}. $f_{j,d}(n_l)$ is $0$ for $j > \frac{n_l + d}{2}$, so it suffices to take the sum for 
\begin{align*}
    j \leq \min\left\{\frac{n_l + d}{2}, n_l\right\} = \frac{n_l + \min\{n_l, d_l\}}{2}.
\end{align*}
Thus, we get that $R(1,n_0)$ is less than or equal to
\begin{align*}
    \sum_{j_1 \leq \frac{n_1 + \min\{n_1, d_1\}}{2}} f_{j_1,d_1}(n_1) \sum_{j_2 \leq \frac{n_2 + \min\{n_2, d_2\}}{2}} f_{j_2,d_2}(n_2) \dots \sum_{j_L = 0}^{\min\{n_L,d_L\}} f_{j_L,d_L}(n_L).
\end{align*}
By rewriting this expression we get the desired upper bound.
\end{proof}

As already mentioned this is a sharper version of the upper bound
\begin{align*}
    \sum_{j_1 \leq \min\{n_1, d_1\}} {n_1 \choose j_1} \sum_{j_2 \leq \min\{n_2, d_2\}} {n_2 \choose j_2} \dots \sum_{j_L \leq \min\{n_L,d_L\}} {n_l \choose j_l},
\end{align*}
which was given in \cite{upperBoundRn}. It is not obvious that the expression given in \Cref{theorem: upper bound R(n)} is smaller, so we will give a quick proof.

\begin{proof}
For $n \geq d$ we can use the Hockey-stick identity to get
\begin{align*}
    \sum_{d \leq j \leq \frac{n+d}{2}} f_{j,d}(n) &= \sum_{d \leq j \leq \frac{n+d}{2}} {n-2j+2d-1 \choose d-1} + {n-2j+2d-2 \choose d-1} \\
    &= \sum_{i=d-1}^{n-1} {i \choose d-1} \\
    &= {n \choose d}.
\end{align*}
This implies
\begin{align*}
    \sum_{j_l \leq \frac{n_l+\min\{n_l,d_l\}}{2}} f_{j_l,d_l}(n_l) \leq \sum_{j_l \leq \min\{d_l,n_l\}} {n_l \choose j_l}.
\end{align*}
Since the inner sums of
\begin{align*}
    \sum_{j_1 \leq \frac{n_1 + \min\{n_1, d_1\}}{2}} f_{j_1,d_1}(n_1) \sum_{j_2 \leq \frac{n_2 + \min\{n_2, d_2\}}{2}} f_{j_2,d_2}(n_2) \dots \sum_{j_L = 0}^{\min\{n_L,d_L\}} f_{j_L,d_L}(n_L).
\end{align*}
are monotonically decreasing in $j_l$, we can apply this estimate to every sum to get an upper bound of
\begin{align*}
    &\sum_{j_1 \leq \min\{n_1, d_1\}} {n_1 \choose j_1} \sum_{j_2 \leq \min\{n_2, d_2\}} {n_2 \choose j_2} \dots \sum_{j_L \leq \min\{n_L,d_L\}} {n_l \choose j_l}. \qedhere
\end{align*}
\end{proof}

\newpage

\section{Estimates for $\mathcal{E}(\mathbf{n})$}
\label{Section 3}

\subsection{$1$-dimensional input}

This section deals with a different kind of expressive power. Instead of looking at the maximal possible number of linear regions, we are interested in calculating any function with a certain number of linear regions. As before, let us assume that $n_{L+1} = 1$. We begin with the case of input dimension $1$. \\

It turns out that width $2$ is not enough to calculate arbitrary piecewise affine functions. This is because a layer of $2$ enhanced neurons can only calculate functions that are either monotone or bounded from above or below due to $\mathcal{R}^m(2) = 2$ (see proof of \Cref{theorem: R(n) in 1 dimension}). \\

Consider the set $S$ of functions that do not satisfy either of those properties. Let $T$ be the set of width $2$ networks that calculate a function in $S$. If $T$ is not empty, we can pick a network of minimum depth in $T$. Let $h$ be the function calculated by this network. Since $h$ is not bounded from above or below, the last layer of the network has to calculate a strictly monotone function. But this means that we can remove the last layer and again get a network in $T$, which is a contradiction to the minimality of depth. Hence, networks of width $2$ can not calculate functions in $S$. \\

Interestingly, networks of width $3$ are capable of computing any piecewise affine function. But it is significantly easier to prove a result for width $\geq 4$, where we can use a different strategy. We will start the bound in the case of width $3$ and present the case $\geq 4$ later on.

\begin{proof}[Proof of \Cref{theorem: E(n) in 1 dimension with width 3}]
We use induction on $L$. \\

\textit{Initial Case:} $L=1$. Pick an arbitrary piecewise affine function with at most $3$ linear regions. Let $a_1$, $a_2$ and $a_3$ be the slopes of the linear regions from left to right. If the piecewise affine function has less than $3$ linear regions we can for example choose $a_1 = a_2$, so this is no restriction. Let $(x_1,y_1)$ and $(x_2,y_2)$ be the breakpoints from left to right. Then the function can be represented by
\begin{align*}
    y_1 - \sgn(a_1) \max\{0,-|a_1|(x-x_1)\} + \sgn(a_2) \max\{0,|a_2|(x-x_1)\} \\
    + \sgn(a_3-a_2) \max\{0,|a_3-a_2|(x-x_2)\}.
\end{align*}
Thus, it can be calculated by a neural network with $1$ hidden layer consisting of $3$ neurons. \\

\textit{Induction step:} $L-1 \rightarrow L$. For any piecewise affine function with at most $L+1$ linear regions we can just add a layer that calculates the identity and use the induction hypothesis. It remains to prove that any piecewise affine function with $L+2$ linear regions can be calculated. Pick any such function. Let $a_1,a_2,\dots,a_{L+2}$ be the slopes and $(x_1,y_1),(x_2,y_2),\dots,(x_{L+1},y_{L+1})$ be the breakpoints from left to right. Define
\begin{equation*}
    y_0 = 
    \begin{cases}
    \infty &\text{if } a_1 < 0, \\
    y_1 &\text{if } a_1 = 0, \\
    -\infty &\text{if } a_1 > 0.
    \end{cases}
\end{equation*}
This is supposed to be the "height of the breakpoint at $-\infty$". Similarly, we can define $y_{L+2}$ in terms of $a_{L+2}$ and $y_{L+1}$. \\

\textit{Case 1:} Assume that there is an $i$ such that $y_{i-1} \leq y_i \leq y_{i+1}$ or $y_{i-1} \geq y_i \geq y_{i+1}$. We only consider $y_{i-1} \leq y_i \leq y_{i+1}$ because the other case is pretty similar. This implies $a_i, a_{i+1} \geq 0$. \\

\textit{Case 1.1:} $a_{i+1} > 0$. Let $g$ be a piecewise affine function with only $L+1$ linear regions constructed in the following way: \\

\begin{figure}[h]
    \centering
    \begin{tikzpicture}
    \begin{axis}[no markers, ymin=-1, ymax=2]
    \addplot+[domain=-0.3:1.3,samples=100,color=blue]{5-3/2*max(0,7*x-1)+max(0,7*x-3)+max(0,7*x-4)-max(0,-7*x+5)-max(0,7*x-5)};
    \addplot+[domain=-0.3:1.3,samples=100,color=blue, dashed ]{5-3/2*max(0,7*x-1)+max(0,7*x-3)+max(0,7*x-3)-max(0,-7*x+5)-2*max(0,7*x-13/3)+max(0,7*x-5)};
    \end{axis}
    \end{tikzpicture}
    \caption{Construction of $g$}
\end{figure}
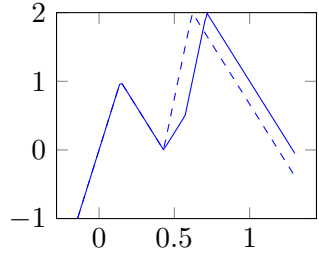

$g$ coincides with $f$ until $(x_{i-1},y_{i-1})$. Then $g$ has slope $a_{i+1}$ and keeps it until it reaches height $y_{i+1}$, which will take place at
\begin{align*}
    \hat{x}_{i+1} &\coloneqq x_{i-1} + \frac{y_{i+1}-y_{i-1}}{a_{i+1}} \\
    &= x_{i-1} + \frac{y_i - y_{i-1}}{a_{i+1}} + \frac{y_{i+1}-y_i}{a_{i+1}} \\
    &= x_{i-1} + \frac{a_i}{a_{i+1}} (x_i-x_{i-1}) + (x_{i+1}-x_i).
\end{align*}
From there onwards $g$ continues like $f$ after its breakpoint at $(x_{i+1},y_{i+1})$. In formulas that means
\begin{align*}
    g(x) =
    \begin{cases}
    f(x) &\text{for } x < x_{i-1}, \\
    y_{i-1} + a_{i+1}(x-x_{i-1}) &\text{for } x_{i-1} \leq x < \hat{x}_{i+1}, \\
    f(x - \hat{x}_{i+1} + x_{i+1}) &\text{for } \hat{x}_{i+1} \leq x.
    \end{cases}
\end{align*}
Furthermore, it is easy to verify that
\begin{equation*}
    h(x) \coloneqq
    \begin{cases}
    x &\text{for } x < x_{i-1}, \\
    x_{i-1} + \frac{a_i}{a_{i+1}} (x-x_{i-1}) &\text{for } x_{i-1} \leq x < x_i, \\
    x + \hat{x}_{i+1} - x_{i+1} &\text{for } x_i \leq x.
    \end{cases}
\end{equation*}
is a piecewise affine function with at most $3$ linear regions. A straightforward calculation now shows that $f = g \circ h$. But $g$ has only $L+1$ linear regions, so it can be calculated by a neural network with $L-1$ layers. $h$ has at most $3$ linear regions and can therefore be calculated by $1$ additional layer, which finishes Case 1.1. \\

\textit{Case 1.2:} $a_{i+1} = 0$. We define
\begin{equation*}
    g(x) \coloneqq
    \begin{cases}
    f(x) &\text{for } x < x_i, \\
    f(x - x_i + x_{i+1}) &\text{for } x_i \leq x
    \end{cases}
\end{equation*}
and
\begin{equation*}
    h(x) \coloneqq
    \begin{cases}
    x &\text{for } x < x_i, \\
    x_i &\text{for } x_i \leq x < x_{i+1}, \\
    x + x_i - x_{i+1} &\text{for } x_{i+1} \leq x.
    \end{cases}
\end{equation*}
It is immediate to see that $f = g \circ h$ and the same reasoning as above finishes Case 1.2. \\

\textit{Case 2:} Assume that there is an $i$ such that $y_{i-1} \leq y_{i+1} \leq y_i \leq y_{i+2}$ or $y_{i-1} \geq y_{i+1} \geq y_i \geq y_{i+2}$ as well as $a_j \neq 0$ for $j \in \{i,i+1,i+2\}$. \\

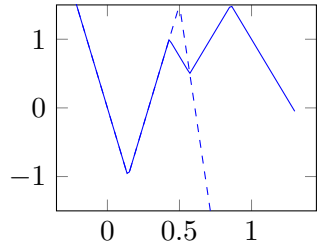
\begin{figure}[h]
    \centering
    \begin{tikzpicture}
    \begin{axis}[no markers, ymin=-1.5, ymax=1.5]
    \addplot+[domain=-0.3:1.3,samples=100,color=blue]{-7*x+2*max(0,7*x-1)-3/2*max(0,7*x-3)+max(0,7*x-4)-max(0,7*x-6)};
    \addplot+[domain=-0.3:1.3,samples=100,color=blue, dashed ]{-7*x+2*max(0,7*x-1)-3*max(0,7*x-7/2)};
    \end{axis}
    \end{tikzpicture}
    \caption{Construction of $g$}
\end{figure}

Similar to Case 1 we only consider $y_{i-1} \leq y_{i+1} \leq y_i \leq y_{i+2}$ and define
\begin{align*}
    g(x) \coloneqq
    \begin{cases}
    f(x) &\text{for } x < x_{i-1}, \\
    y_{i-1} + a_i (x-x_{i-1}) &\text{for } x_{i-1} \leq x < \hat{x}_{i+2}, \\
    f(\frac{a_i}{a_{i+2}} (x-\hat{x}_{i+2})+x_{i+2}) &\text{for } \hat{x}_{i+2} \leq x
    \end{cases}
\end{align*}
and
\begin{equation*}
    h(x) \coloneqq
    \begin{cases}
    x &\text{for } x < x_i, \\
    x_i + \frac{a_{i+1}}{a_i}(x-x_i) &\text{for } x_i \leq x < x_{i+1}, \\
    \frac{a_{i+2}}{a_i} (x - x_{i+2}) + \hat{x}_{i+2} &\text{for } x_{i+1} \leq x,
    \end{cases}
\end{equation*}
where
\begin{align*}
    \hat{x}_{i+2} &\coloneqq x_{i-1} + \frac{y_{i+2}-y_{i-1}}{a_i} \\
    &= x_{i-1} + \frac{y_i-y_{i-1}}{a_i} + \frac{y_{i+1}-y_i}{a_i} + \frac{y_{i+2}-y_{i+1}}{a_i} \\
    &= x_{i-1} + (x_i-x_{i-1}) + \frac{a_{i+1}}{a_i} (x_{i+1}-x_i) + \frac{a_{i+2}}{a_i} (x_{i+2}-x_{i+1}) \\
    &= x_i + \frac{a_{i+1}}{a_i} (x_{i+1}-x_i) + \frac{a_{i+2}}{a_i} (x_{i+2}-x_{i+1}).
\end{align*}
Again, a straightforward calculation shows that $g$ and $h$ are continuous and $f = g \circ h$. Moreover, $g$ has $L$ linear regions and $h$ has at most $3$, which finishes Case 2. \\

It remains to show that we are always in one of those two cases, so let's assume the opposite. As we are not in Case 1, we have
\begin{align*}
    y_0 &< y_1 > y_2 < y_3 > y_4 < \dots \quad \text{or}\\
    y_0 &> y_1 < y_2 > y_3 < y_4 > \dots
\end{align*}
Wlog assume the first line holds, otherwise we can just flip all inequalities. Note that this implies $a_j \neq 0$ and that $y_{L+2}$ cannot be equal to $y_{L+1}$, so it must be $\infty$ or $-\infty$. For the same reason we have $y_0 = -\infty$. \\

Since we are also not in Case 2, we get $y_1 > y_3$ because otherwise we would have
\begin{equation*}
    -\infty = y_0 \leq y_2 \leq y_1 \leq y_3.
\end{equation*}
This implies $y_2 < y_4$ because otherwise we would have
\begin{equation*}
    y_1 \geq y_3 \geq y_2 \geq y_4.
\end{equation*}
By induction we get
\begin{align*}
    y_1 &> y_3 > y_5 > y_7 > \dots \quad \text{and} \\
    y_2 &< y_4 < y_6 < y_8 < \dots 
\end{align*}
If $L+2$ is even, we can conclude that
\begin{align*}
    y_2 < y_{L+2} < y_{L+1}.
\end{align*}
If $L+1$ is odd, we obtain
\begin{align*}
    y_1 > y_{L+2} > y_{L+1}.
\end{align*}
In particular, $y_{L+2}$ cannot be equal to $\infty$ or $-\infty$ either, which is a contradiction.
\end{proof}

In this proof it is worth noting that we always composed our intermediate result with another function, which was calculated by an extra layer. A less convoluted approach would be to just write the piecewise affine function as a sum and calculate the different summands separately. The downside of this idea is that we need $4$ neurons in each layer to store the input and intermediate result. The upside is that we can use all of the remaining neurons to construct new linear regions. We make these ideas precise in the proof of \Cref{theorem: E(n) in 1 dimension}.

\begin{proof}[Proof of \Cref{theorem: E(n) in 1 dimension}]
In each layer we use two neurons to save the input via
\begin{equation*}
    x = \max\{x,0\} + \max\{-x,0\}.
\end{equation*}
With the same trick we also store our intermediate results. This means that we need $4$ neurons in each layer to save data and the remaining neurons are used for calculations. \\

Observe that we can write any piecewise affine function $f$ with $k$ linear regions as the sum of outputs of $k$ different neurons. This can be done by just working through the breakpoints from left to right and adding the difference in slopes. \\

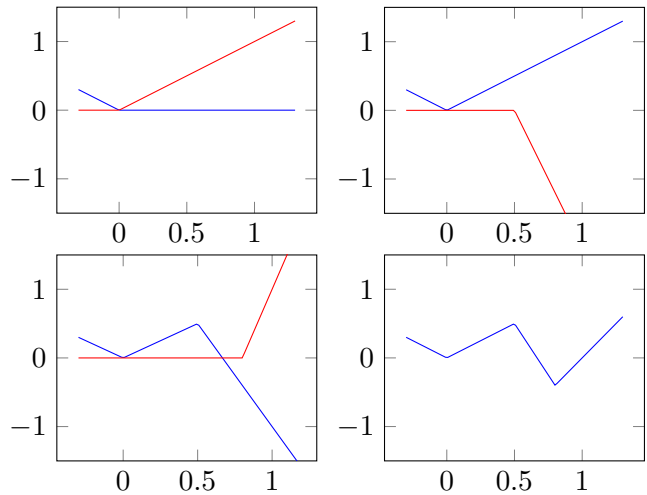
\begin{figure}[h]
    \centering
    \begin{tikzpicture}
    \begin{axis}[no markers, ymin=-1.5, ymax=1.5]
    \addplot+[domain=-0.3:1.3,samples=100,color=blue]{max(0,-x)};
    \addplot+[domain=-0.3:1.3,samples=100,color=red]{max(0,x)};
    \end{axis}
    \end{tikzpicture}
    \begin{tikzpicture}
    \begin{axis}[no markers, ymin=-1.5, ymax=1.5]
    \addplot+[domain=-0.3:1.3,samples=100,color=blue]{max(0,-x)+max(0,x)};
    \addplot+[domain=-0.3:1.3,samples=100,color=red]{-max(0,4*x-2)};
    \end{axis}
    \end{tikzpicture}
    \begin{tikzpicture}
    \begin{axis}[no markers, ymin=-1.5, ymax=1.5]
    \addplot+[domain=-0.3:1.3,samples=100,color=blue]{max(0,-x)+max(0,x)-max(0,4*x-2)};
    \addplot+[domain=-0.3:1.3,samples=100,color=red]{max(0,5*x-4)};
    \end{axis}
    \end{tikzpicture}
    \begin{tikzpicture}
    \begin{axis}[no markers, ymin=-1.5, ymax=1.5]
    \addplot+[domain=-0.3:1.3,samples=100,color=blue]{max(0,-x)+max(0,x)-max(0,4*x-2)+max(0,5*x-4)};
    \end{axis}
    \end{tikzpicture}
    \caption{Constructing the linear regions of $f$ from left to right}
\end{figure}

These neurons use $x$ as the input and their output is added to the intermediate result. Hence, the output will be the sum of the $k$ neurons, which is equal to $f$. Finally, we can also use the two neurons storing the intermediate result in the first layer for calculations, which explains the $+2$ in the final result.
\end{proof}

\newpage

\subsection{$n_0$-dimensional input}

For the case of $n_0$-dimensional input we use a similar strategy as in the proof of \Cref{theorem: E(n) in 1 dimension}: We try to write the piecewise affine function as a sum of simpler functions and calculate them individually. Our first step is to give a decomposition of an affine function defined on a simplex. It is well-known and was for example used in \cite{approximationSobolevFunctions}.

\begin{lemma}[\cite{approximationSobolevFunctions}, Section 2]
\label{barycentric decomposition}
Let $\Delta \subset \mathbb{R}^d$ be a simplex with vertices $x_0,\dots,x_d$ and $f: \mathbb{R}^d \to \mathbb{R}$ be any affine function. Then there are affine functions $f_0,\dots,f_d: \mathbb{R}^d \to \mathbb{R}$ such that
\begin{equation*}
    f = \sum_{i=0}^d f_i,
\end{equation*}
$f_i|_F$ only depends on $f|_F$ for any face $F$ of $\Delta$ and we have $f_i = 0$ on the face opposite to $x_i$.
\end{lemma}

\begin{proof}
The idea is to use the barycentric coordinates $(a_0,\dots,a_d)$ of $x$ with respect to $\Delta$, i.e. $\sum_{i=0}^d a_i = 1$ and
\begin{equation*}
    x = \sum_{i=0}^d a_i x_i.
\end{equation*}
The $a_i$ are uniquely characterised by these equations, so the functions
\begin{equation*}
    f_i(x) \coloneqq a_i f(x_i)
\end{equation*}
are affine. Note that $a_i = 0$ on the face opposite to $x_i$, so the same holds true for $f_i$. More generally, let $F$ be any face of $\Delta$ and $x \in F$. Then we get $a_i = 0$ if $x_i \not\in F$, so $f_i|_F$ only depends on $f|_F$. Finally, we have
\begin{equation*}
    f(x) = f\left(\sum_{i=0}^d a_i x_i\right) = \sum_{i=0}^d a_i f(x_i) = \sum_{i=0}^d f_i(x)
\end{equation*}
since $f$ is affine.
\end{proof}

Ideally, parts of our neural network would calculate the summands used in \Cref{barycentric decomposition} for every linear region and simply add them up. The problem is that the summands would have to be $0$ outside of the simplex, which is not possible because neural networks with ReLU as their activation function will always compute continuous functions. However, the summands are non-zero in one vertex of the simplex. The solution is to piece all summands together that are non-zero at one particular vertex $p$. This will give a function $f_p$ with the shape of a cone. $f_p$ is zero outside of the simplices adjacent to the vertex $p$. \\

It remains to calculate this function with the help of a neural network. We will do this by decomposing it into simpler functions once again. These functions are supposed to coincide with $f_p$ along one edge adjacent to $p$ up to an additive constant. Furthermore, they are $0$ along any other edge adjacent to $p$. The computation of those functions is encapsulated in the following lemma.

\begin{lemma}\label{lego lemma}
For $i=1,\dots,k$ let $S_i: a_i \cdot x = 0$ be $d-1$-dimensional hypersurfaces going through the origin. Assume that they are the faces of the cone
\begin{equation*}
    C \coloneqq \{x \in \mathbb{R}^d : a_i \cdot x \geq 0 \text{ for } i=1,\dots,k\}.
\end{equation*}
Let $v \in C^\mathrm{o}$, $L \coloneqq \{rv : r \geq 0\}$ and $g: L \to \mathbb{R}$ be any linear function. Define 
\begin{align*}
    C_i &\coloneqq \conv((C \cap S_i) \cup L).
\end{align*}
Let $f$ be the piecewise affine function with linear regions $C_1^\circ,\dots,C_k^\circ$ and $C^\mathrm{c}$, $f|_L = g$ and $f|_{C^\mathrm{c}} = 0$. Then $k$ layers consisting of $2$ neurons can calculate $f$ if you save the input somewhere else.
\end{lemma}

\begin{proof}
First of all, we scale the $a_i$ such that $a_i \cdot v = 1$. Observe that
\begin{equation*}
    f(x) = s \max\{0, \min\{a_i \cdot x : i=1,\dots,k\}\},
\end{equation*}
where $s \in \mathbb{R}$ is determined by the slope of $g$. This is because the linear region is determined by the index that minimizes $a_i \cdot x$. The boundaries of those linear regions are determined by $a_i \cdot x = 0$ and $a_i \cdot x = a_j \cdot x$, respectively. Since $v$ lies on the boundaries of the latter type, we can conclude that the function coincides with $f$. Next up, for $a,b \in \mathbb{R}$ we have
\begin{align*}
    \min\{a,b\} &= \frac{1}{2} (a+b-|a-b|) \\
    &= \frac{1}{2} (a+b-\max\{a-b,0\}-\max\{-a+b,0\})
\end{align*}
so the minimum of two numbers can be calculated with $1$ layer consisting of $2$ enhanced neurons if $a$ and $b$ are stored somewhere else. Observe that in our case $a$ and $b$ are a linear function of the input $x$. Using
\begin{equation*}
    \min\{a,b,c\} = \min\{\min\{a,b\},c\}
\end{equation*}
for $c \in \mathbb{R}$ we need $k-1$ layers consisting of $2$ neurons to calculate
\begin{align*}
    \min\{a_i \cdot x : i=1,\dots,k\}.
\end{align*}
Hence, $k$ layers consisting of $2$ neurons can calculate $f$.
\end{proof}

This shows that we can calculate the essential building blocks of any piecewise affine function. We just need to put them together in order to prove that we can calculate any piecewise affine function with a neural network.

\begin{proof}[Proof of \Cref{theorem: lower bound E(n)}]
Let $f$ be any piecewise affine function and $V$ be the number of vertices of the linear regions of $f$. Consider the graph of $f$, which is in $\mathbb{R}^{n_0+1}$. Extend the images of the linear regions in the graph to $n_0$-dimensional hyperplanes. These hyperplanes define at most ${k \choose n_0+1}$ intersection points because we get at most one point for any collection of $n_0+1$ hyperplanes. Each vertex of a linear region a projection of one of those intersection points onto $\mathbb{R}^{n_0}$, which means that
\begin{equation*}
    V \leq {k \choose n_0+1}.
\end{equation*}
Now, we triangulate the linear regions of $f$. Let $S$ be the number of $n_0$-dimensional simplices in the triangulation. Then we have
\begin{equation*}
    S \leq {V \choose n_0+1}.
\end{equation*}
For any vertex $p$ of the triangulation define
\begin{equation*}
    f_p = \sum_{\Delta} \mathbbm{1}_\Delta f_{p,\Delta},
\end{equation*}
where the sum runs over all simplices $\Delta$ adjacent to $p$ and $f_{p,\Delta}$ is taken from \Cref{barycentric decomposition}. These functions fit together to form a continuous piecewise affine function since the $f_{p,\Delta}$ agree on common faces by \Cref{barycentric decomposition}. The graph of $f_p$ is a cone with the union of the simplices as the base and $(p,f(p))$ as the apex. Observe that
\begin{align*}
    \sum_p f_p &= \sum_p \sum_{\Delta} \mathbbm{1}_\Delta f_{p,\Delta} \\
    &= \sum_{\Delta} \mathbbm{1}_\Delta \sum_p f_{p,\Delta} \\
    &= \sum_{\Delta} \mathbbm{1}_\Delta f \\
    &= f.
\end{align*}
It remains to calculate $f_p$ with the help of neurons. Wlog let $p$ be the origin. For any edge $e$ containing $p$ we apply \Cref{lego lemma} with $g = f_p - f_p(0)$ and $C_i$ being extended versions of the simplices in $M'$ adjacent to $p$. Let $f_{p,e}$ be the functions that is constructed in \Cref{lego lemma}. Since
\begin{equation*}
   \max\left\{f_p(0) + \sum_e f_{p,e},0\right\} 
\end{equation*}
coincides with $f_p$ on every edge $e$, has the same linear regions and the same support, we know that it must be equal to $f_p$. Let $E_p$ and $S_p$ be the number of edges and simplices adjacent to $p$, respectively. Since any simplex adjacent to $p$ occurs in the number of layers in \Cref{lego lemma} for each of its $n_0$ edges adjacent to $p$, we need $n_0 S_p$ layers consisting of $2$ neurons to calculate all of the $f_{p,e}$. Furthermore, we need $2$ additional neurons in each layer to save the intermediate results of the sum $\sum_e f_{p,e}$. Finally, $1$ final layer can calculate $f_p$, which means that we can calculate $f_p$ in $1 + n_0 S_p$ layers consisting of $4$ neurons. This implies that we need
\begin{equation*}
    \sum_p 1 + n_0 S_p = V + n_0(n_0+1)S = O(k^{(n_0+1)^2})
\end{equation*}
layers to calculate $f$. Each layer consists of $4$ neurons to calculate $f_p$, $2$ neurons to keep track of the sum $\sum_p f_p$ and $2n_0$ neurons to save the input, so it has $2n_0+6$ neurons in total.
\end{proof}

\newpage

\section{Universal Approximation Theorems}

One straightforward application of expressive powers are so-called universal approximation theorems. The idea is to prove that neural networks can approximate any function of a certain class. For the class of continuous functions this can for example be done by first approximating polynomials and then using the fact that polynomials are dense \cite{universalApproximation}. We will apply this idea with piecewise affine functions instead of polynomials.

\begin{proof}[Proof of \Cref{corollary: lower bound E(n) universal approximation}]
This follows directly from \Cref{theorem: lower bound E(n)} and the fact that piecewise affine functions are dense in those spaces.
\end{proof}

A major drawback of universal approximation theorems like \Cref{corollary: lower bound E(n) universal approximation} is that they do not make any quantitative statement about how fast we can approximate a function. One result in this direction is given for example in \cite{optimalApproximationRate}, where they used piecewise constant functions in the approximation. The advantage of using piecewise affine functions rather than piecewise constant functions is that we can also give an approximation rate in the $W^{1,1}$-norm. Suitable estimates can be found in \cite{approximationSobolevFunctions}.

\begin{lemma}[\cite{approximationSobolevFunctions}, Proof of Theorem 1]
\label{lemma: affine interpolation}
Let $p \in [1,\infty)$. Pick a function $f \in W^{1,p}(\mathbb{R}^{n_0})$. Consider a triangulation $T$ of $\mathbb{R}^{n_0}$ with simplices of diameter $<r$. Define $T_h$ as the translation of $T$ by $h$. Let $g_h$ be a piecewise affine function that is affine on each simplex of $T_h$ and coincides with $f$ on the vertices of $T_h$. Then we have
\begin{align*}
\int_{B_r} \int_{\mathbb{R}^{n_0}} |Df - Dg_h| \, \mathrm{d}x \, \mathrm{d}h \leq C \int_{B_r} \int_{\mathbb{R}^{n_0}} |Df(x) - Df(x+h)| \, \mathrm{d}x \, \mathrm{d}h.
\end{align*}
\end{lemma}

It remains to construct a triangulation. Then we just need to calculate the function corresponding to that triangulation with a neural network in order to prove \Cref{theorem: approximation rate in sobolev norm}.

\begin{proof}[Proof of \Cref{theorem: approximation rate in sobolev norm}]
Define $g_h$ as in \Cref{lemma: affine interpolation}. Using Poincar\'e's inequality, \Cref{lemma: affine interpolation} and the Lipschitz continuity we get
\begin{align*}
\fint_{B_r} \|f-g_h\|_{W^{1,1}} \, \mathrm{d}h &\leq C_0 \fint_{B_r} \int_{\mathbb{R}^{n_0}} |Df - Dg_h| \, \mathrm{d}x \, \mathrm{d}h \\
&\leq C_1 \fint_{B_r} \int_{[-r,1+r]^{n_0}} |Df(x) - Df(x+h)| \, \mathrm{d}x \, \mathrm{d}h \\
&\leq C_1 (1+2r)^{n_0} \Lip(Df) r \\
&\leq C_2 r.
\end{align*}
In particular, there is an $h \in B_r$ such that
\begin{align*}
\|f-g_h\|_{W^{1,1}} \leq C_2 r.
\end{align*}
Let $g \coloneqq g_h$. We need a triangulation with diameter $r = \frac{\varepsilon}{C_2}$ to get an approximation up to an error of $\varepsilon > 0$. To construct this triangulation we subdivide $[0,1]^{n_0}$ into $\lceil\frac{\sqrt{n_0}}{r}\rceil^{n_0}$ cubes of side length $\leq \frac{r}{\sqrt{n_0}}$. Each cube can be subdivided into $n_0!$ simplices of the desired diameter in the following way: For any permutation $\pi$ of the numbers $1,\dots,d$ we define
\begin{align*}
(v^\pi_i)_j \coloneqq \mathbbm{1}_{\pi^{-1}(j) \leq i}
\end{align*}
We decompose the unit cube into simplices defined by
\begin{align*}
    \Delta^\pi \coloneqq \conv\{v^\pi_i: i=0,\dots,d\}.
\end{align*}
Scaling and translation gives the desired triangulation. Altogether, this triangulation needs
\begin{align*}
    \lceil\frac{\sqrt{n_0}}{r}\rceil^{n_0} n_0!
\end{align*}
simplices. As shown in the proof of \Cref{theorem: lower bound E(n)} a neural network needs $O(n_0(n_0+1))$ layers per simplex to calculate $g$. All in all, we need
\begin{align*}
    O(\varepsilon^{-n_0} C_2^{n_0} n_0^{\frac{n_0}{2}+1} (n_0+1)!)
\end{align*}
layers to approximate $f$ up to an error of $\varepsilon$ in the $W^{1,1}$-norm. This gives an approximation rate of
\begin{align*}
    &C \Lip(Df) L^{-\frac{1}{n_0}}. \qedhere
\end{align*}
\end{proof}

\newpage

\section{Width vs Depth}

\subsection{Width Inefficiency}

Another application of expressive powers are results that compare width and depth efficiency of neural networks. It turns out that wide and shallow networks sometimes need exponentially many neurons to calculate the same function as a deep network, which is the content of \Cref{theorem: width inefficiency in 1 dimension}.

\begin{proof}[Proof of \Cref{theorem: width inefficiency in 1 dimension}]
The function $h: \mathbb{R} \to \mathbb{R}$ defined by
\begin{equation*}
    h(x) \coloneqq 1 - \max\{0,1-3x\} - \max\{0,3x-1\} + \max\{0,6x-4\}
\end{equation*}
can be calculated by $1$ layer of $3$ neurons, so $L^2$ hidden layers can calculate the composition of $L^2$ copies of $h$. We will call this composition $g$. \\

\begin{figure}[h]
    \centering
    \begin{tikzpicture}
    \begin{axis}[no markers, ymin=-1, ymax=2]
    \addplot+[domain=-0.3:1.3,samples=100,color=blue]{1-max(1-3*x,0)-max(3*x-1,0)+max(6*x-4,0)};
    \addplot+[dashed,color=black] coordinates {(0,0)(0,1)};
    \addplot+[dashed,color=black] coordinates {(0,1)(1,1)};
    \addplot+[dashed,color=black] coordinates {(1,1)(1,0)};
    \addplot+[dashed,color=black] coordinates {(1,0)(0,0)};
    \end{axis}
    \end{tikzpicture}
    \begin{tikzpicture}[
        evaluate={
            function f(\x) {
                return 1-max(1-3*\x,0)-max(3*\x-1,0)+max(6*\x-4,0);
            };
        },
    ]
    \begin{axis}[no markers,ymin=-1,ymax=2]
    \addplot+[domain=-0.3:1.3,samples=200,color=blue]{f(f(x))};
    \addplot+[dashed,color=black] coordinates {(0,0)(0,1)};
    \addplot+[dashed,color=black] coordinates {(0,1)(1,1)};
    \addplot+[dashed,color=black] coordinates {(1,1)(1,0)};
    \addplot+[dashed,color=black] coordinates {(1,0)(0,0)};
    \end{axis}
    \end{tikzpicture}
    \caption{Graph of $h$ and $h \circ h$}
\end{figure}
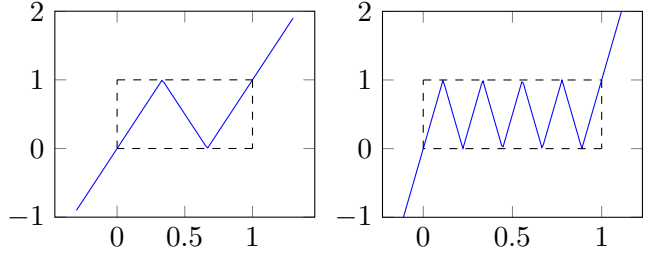

Similar to the construction in \Cref{theorem: R(n) in 1 dimension}, $g$ has $3^{L^2}$ linear regions of length $3^{-L^2}$. Consider the intervals
\begin{equation*}
    I_i \coloneqq \left(\frac{i-\frac{1}{2}}{3^{L^2}},\frac{i+\frac{1}{2}}{3^{L^2}}\right)
\end{equation*}
for $i = 1,\dots,3^{L^2}-1$. $g-\frac{1}{2}$ has positive sign in $I_1$, negative sign in $I_2$ and so on. \\

Now, let $f$ be any function that can be calculated by a neural network of depth $L$ and width $\leq N$. We say that $f$ is a good approximation for $g$ in one of those intervals if $f-\frac{1}{2}$ has the same sign as $g-\frac{1}{2}$ at least once (and is non-zero there). If $f$ is a good approximation in two consecutive intervals, it has to cross the line $y = \frac{1}{2}$. By \Cref{theorem: R(n) in 1 dimension} $f$ has at most $(N+1)^L$ linear regions, so it can cross this line at most $(N+1)^L$ times. Thus, there are at most $(N+1)^L$ pairs of consecutive intervals where $f$ is a good approximation. \\

Altogether, there are $3^{L^2}-2$ pairs of consecutive intervals, so there are at least $3^{L^2}-(N+1)^L-2$ pairs of intervals such that $f$ is not a good approximation in at least one of them. Since one interval is part of at most two pairs, there are at least
\begin{equation*}
    \frac{1}{2} (3^{L^2}-(N+1)^L-2)
\end{equation*}
intervals where $f$ is not a good approximation. For any $j$ such that $f$ is not a good approximation in $I_j$ we have
\begin{equation*}
    \int_{I_j} |f-g| \, \mathrm{d}x \geq \int_{I_j} \left|g-\frac{1}{2}\right| \, \mathrm{d}x = \frac{1}{4} 3^{-L^2}.
\end{equation*}
All in all, this implies
\begin{equation*}
    \int_0^1 |f-g| \, \mathrm{d}x \geq \frac{1}{2} (3^{L^2}-(N+1)^L-2) \cdot \frac{1}{4} 3^{-L^2}.
\end{equation*}
In case of $N < 3^{L-1}-1$ and $L \geq 2$ this can be bounded by
\begin{align*}
    \int_0^1 |f-g| \, \mathrm{d}x &\geq \frac{1}{8} \frac{3^{L^2}-(3^{L-1}-1)^L-2}{3^{L^2}} \\
    &\geq \frac{1}{8} \frac{3^{L^2}-3^{(L-1)L}}{3^{L^2}} \\
    &= \frac{1}{8} (1-3^{-L}) \\
    &\geq \frac{1}{9}.
\end{align*}
Hence, $g$ cannot be approximated by any neural network of depth $L$ unless it has exponential width.
\end{proof}

\newpage

\subsection{Depth Efficiency}

After we have seen the inefficiencies of wide networks, it is natural to ask the opposite question: Can deep networks also struggle to approximate functions calculated by wide networks? This is discussed in \Cref{theorem: depth efficiency}.

\begin{proof}[Proof of \Cref{theorem: depth efficiency}]
First of all, we will use \Cref{theorem: upper bound R(n)} to bound the number of linear regions of the function calculated by a network of width $N$ and depth $L$. Each $d_l$ is obviously bounded by $n_0$ and $f_{j,d}(n)$ is bounded by $n \choose j$. This gives an upper bound of
\begin{align*}
    \left(\sum_{j=0}^{n_0} {N \choose j}\right)^L.
\end{align*}
A straightforward induction shows that
\begin{align*}
    \sum_{j=0}^{n_0} {N \choose j} \leq 1 + N^{n_0}
\end{align*}
due to
\begin{align*}
    \sum_{j=0}^{n_0} {N \choose j} &\leq 1 + N^{n_0-1} + {N \choose n_0} \\
    &\leq 1 + N^{n_0-1} + N^{n_0-1}(N-1) \\
    &= 1 + N^{n_0}.
\end{align*}
Hence, the number of linear regions is bounded by $(1+N^{n_0})^L$. By \Cref{theorem: lower bound E(n)} there is a neural network with $O(N^{Ln_0(n_0+1)^2})$ hidden layers of size $2n_0+6$ that can calculate this function.
\end{proof}

\newpage

\section{Approximating Activation Functions}
\label{Section 6}

\subsection{Approximating ReLU with ReLU-computing Functions}

While our ultimate goal is to approximate ReLU networks, we will start with some easier tasks. Approximating the identity allows us to store intermediate results.

\begin{lemma}\label{lemma: identity}[\cite{universalApproximation}, Lemma 4.1]
Assume that $\rho \in C(\mathbb{R})$ is differentiable at at least one point with nonzero derivative. Then a single enhanced neuron with activation function $\rho$ can uniformly approximate the identity on a compact subset with arbitrarily small error.
\end{lemma}

In addition to storing results, we also need a way to calculate new results. This will be done with an approximation of $x \mapsto x^2$.

\begin{lemma}\label{lemma: square}[\cite{universalApproximation}, Proof of Proposition 4.11]
Let $\rho \in C(\mathbb{R})$ be $C^3$ in a neighborhood of a point $\alpha \in \mathbb{R}$ and $\rho''(\alpha) \neq 0$. Then one layer of two enhanced neurons with activation function $\rho$ can approximate $x \mapsto x^2$ with arbitrary precision. 
\end{lemma}

\begin{proof}
For $h > 0$ define $\sigma_h: \mathbb{R} \to \mathbb{R}$ by
\begin{align*}
    \sigma_h(x) = \frac{\rho(\alpha + hx) - 2\rho(\alpha) + \rho(\alpha - hx)}{h^2 \rho''(\alpha)}.
\end{align*}
Obviously, $\sigma_h$ can be calculated by one layer of two enhanced neurons. By Taylor we have
\begin{align*}
    \sigma_h(x) &= \frac{\rho(\alpha) + hx\rho'(\alpha) + \frac{1}{2}h^2x^2\rho''(\alpha) + O(h^3x^3)}{h^2 \rho''(\alpha)} - \frac{2\rho(\alpha)}{h^2 \rho''(\alpha)} \\
    &+ \frac{\rho(\alpha) - hx\rho'(\alpha) + \frac{1}{2}h^2x^2\rho''(\alpha) + O(h^3x^3)}{h^2 \rho''(\alpha)} \\
    &= x^2 + O(hx^3). \qedhere
\end{align*}
\end{proof}

Observe that ReLU-computing functions satisfy the conditions of those two lemmas. Next up, we will explore how the map $x \mapsto x^2$ can be used for calculations. One application is the product of two numbers.

\begin{lemma}\label{lemma: product}[\cite{universalApproximation}, Lemma 4.2]
One layer consisting of two enhanced neurons with activation function $x \mapsto x^2$ can calculate the function $(x,y) \mapsto xy$.     
\end{lemma}

\begin{proof}
This follows directly from
\begin{align*}
    xy &= \frac{1}{4}((x+y)^2-(x-y)^2). \qedhere
\end{align*}
\end{proof}

On a compact interval we can also use the approximation of $x \mapsto x^2$ to approximate inverses quite well.

\begin{lemma}\label{lemma: inverse}[adaptation of \cite{universalApproximation}, Lemma 4.5]
Let $\varepsilon > 0$. Then $m$ layers of three enhanced neurons with square activation function can approximate the function $g: [\varepsilon,2-\varepsilon] \to \mathbb{R}, x \mapsto \frac{1}{x}$ with an error of at most $\frac{(1-\varepsilon)^{2^{m+1}}}{\varepsilon}$.
\end{lemma}

\begin{proof}
Let $z = 1-x$. Observe that $z \in [-1+\varepsilon, 1-\varepsilon]$ and
\begin{align*}
    \frac{1}{1-z} - \prod_{i=0}^n (1+z^{2^i}) &= \frac{1}{1-z} - \sum_{i=0}^{2^{n+1}-1} z^i \\
    &= \sum_{i=2^{n+1}}^\infty z^i \\
    &= \frac{z^{2^{n+1}}}{1-z}
\end{align*}
so it suffices to calculate $\prod_{i=0}^n (1+z^{2^i})$ with the neural network. In order to do that, we use one neuron in the $k$th layer to calculate $z^{2^k} = (z^{2^{k-1}})^2$. With the other two neurons we just multiply the intermediate result $\prod_{i=0}^{k-1} (1+z^{2^i})$ by $(1+z^{2^k})$.
\end{proof}

The next objective is the key step in the proof of \Cref{theorem: ReLU-computing calculates relu}. Namely, we will approximate the ReLU function using a neural network with a ReLU-computing activation function.

\begin{lemma}\label{lemma: approximate ReLU}
Let $\rho: \mathbb{R} \to \mathbb{R}$ be ReLU-computing. A neural network of depth $O(n^\frac{3}{2})$ and width $8$ with activation function $\rho$ can approximate $\max\{0,x\}$ for $x \in [-1,1]$ with an error of $O(e^{-\sqrt{n}})$.
\end{lemma}

\begin{proof}
Let $\xi = \exp(-1/\sqrt{n})$ and $P(x) = \prod_{k=0}^{n-1} (x+\xi^k)$. Define
\begin{align*}
    R(x) = x \frac{P(x)}{P(x)+P(-x)}
\end{align*}
By \cite{absoluteValueApproximation}, Theorem (A) we have
\begin{align*}
    |\max\{0,x\}-R(x)| \leq \frac{3}{2}e^{-\sqrt{n}}
\end{align*}
for $x \in [-1,1]$. Thus, it remains to approximate $R$ with the help of a neural network. First of all, we will use $1$ enhanced neuron in each layer except the last two to save the value of $x$ using \Cref{lemma: identity}. In the first $n$ layers we can approximate $P(x)$ and $P(-x)$ arbitrarily well as follows: $4$ enhanced neurons calculate $x^k$ in the $k$th layer by multiplying $x$ and $x^{k-1}$ from the previous layer with \Cref{lemma: square} and \Cref{lemma: product}. In the other $2$ enhanced neurons we add up and save intermediate values for $P(x)$ and $P(-x)$ using \Cref{lemma: identity}. \\

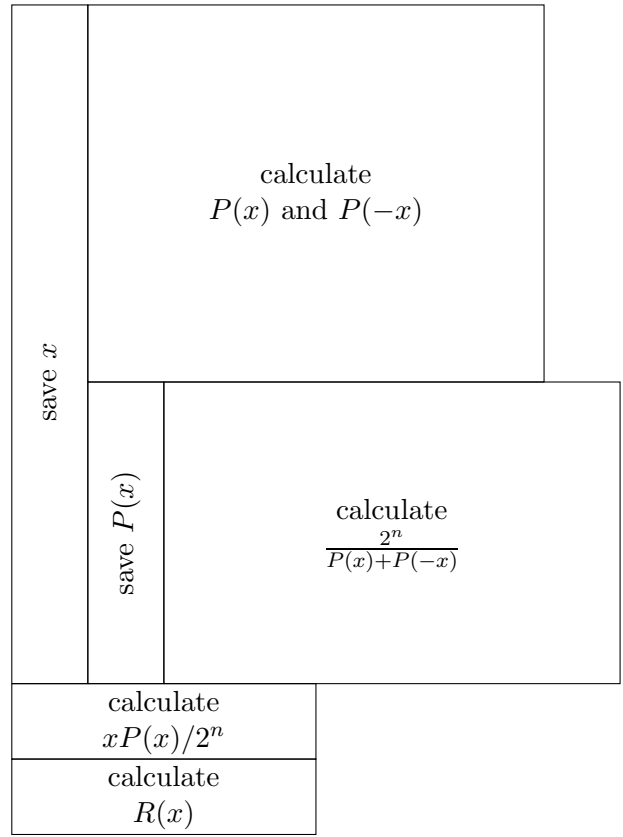
\begin{figure}[h]
\centering
\begin{tikzpicture}
\draw[] (0,1) -- (1,1) -- (1,10) -- (0,10) -- (0,1);
\node[rotate = 90] at (0.5,5) {save $x$}; 
\draw[] (1,1) -- (2,1) -- (2,5) -- (1,5) -- (1,1);
\node[rotate = 90] at (1.5,3) {save $P(x)$}; 
\draw[] (0,1) -- (4,1) -- (4,0) -- (0,0) -- (0,1);
\node[] at (2,0.5) {\begin{tabular}{c} calculate \\ $xP(x)/2^n$ \end{tabular}}; 
\draw[] (0,0) -- (4,0) -- (4,-1) -- (0,-1) -- (0,0);
\node[] at (2,-0.5) {\begin{tabular}{c} calculate \\ $R(x)$ \end{tabular}}; 
\draw[] (1,5) -- (7,5) -- (7,10) -- (1,10) -- (1,5);
\node[] at (4,7.5) {\begin{tabular}{c} calculate \\ $P(x)$ and $P(-x)$ \end{tabular}}; 
\draw[] (2,1) -- (8,1) -- (8,5) -- (2,5) -- (2,1);
\node[] at (5,3) {\begin{tabular}{c} calculate \\ $\frac{2^n}{P(x)+P(-x)}$ \end{tabular}}; 
\end{tikzpicture}
\caption{Schematic Depiction of the Network Architecture}
\end{figure}

For the remaining layers except the last two we use $1$ enhanced neuron to save $P(x)$. Our next goal is to approximate
\begin{align*}
    \frac{2^n}{P(x)+P(-x)}
\end{align*}
using \Cref{lemma: inverse}, so we need some bounds on $z \coloneqq 1-\frac{P(x)+P(-x)}{2^n}$. Observe that for $x \in [\xi^n,1]$
\begin{align*}
    P(x) &= 2x \prod_{k=1}^{n-1} (x+\xi^k) + (-x+\xi^0) \prod_{k=1}^{n-1} (x+\xi^k) \\
    &\geq 2x \prod_{k=1}^{n-1} \xi^k + (-x+\xi^0) \prod_{k=1}^{n-1} |-x+\xi^k| \\
    &\geq 2\xi^n \prod_{k=1}^{n-1} \xi^k + \prod_{k=0}^{n-1} |-x+\xi^k| \\
    &= 2\xi^{n+1 \choose 2} + |P(-x)| \\
    \Rightarrow P(x) + P(-x) &\geq P(x) - |P(-x)| \\
    &\geq 2\xi^{n+1 \choose 2}.
\end{align*}
For $x \in [0,\xi^n]$ every factor of $P(-x)$ is positive, so
\begin{align*}
    P(x) + P(-x) &\geq P(x) \\
    &\geq \prod_{k=0}^{n-1} \xi^k \\
    &\geq \xi^{n \choose 2} \\
    &> 2\xi^{n+1 \choose 2},
\end{align*}
where the last inequality follows from
\begin{align*}
    \xi^n = e^{-\sqrt{n}} \leq e^{-1} < \frac{1}{2}.
\end{align*}
Since $P(x) + P(-x)$ is an even function, we can conclude that the inequality
\begin{align*}
    P(x) + P(-x) \geq 2\xi^{n+1 \choose 2}
\end{align*}
holds for every $x \in [-1,1]$. In addition to that, we have
\begin{align*}
    P(x) + P(-x) &\leq |P(x)| + |P(-x)| \\
    &= (x+1) \prod_{k=1}^{n-1} |x+\xi^k| + (-x+1) \prod_{k=1}^{n-1} |-x+\xi^k| \\
    &\leq (x+1) \prod_{k=1}^{n-1} 2 + (-x+1) \prod_{k=1}^{n-1} 2 \\
    &= 2^n.
\end{align*}
Altogether, this implies
\begin{align*}
    0 \leq z \leq 1 - \frac{\xi^{n+1 \choose 2}}{2^{n-1}}.
\end{align*}
By \Cref{lemma: inverse} we get an approximation with an error of at most
\begin{align*}
    \frac{z^{2^{m+1}}}{1-z} &\leq \frac{2^{n-1}}{\xi^{n+1 \choose 2}} \left(1 - \frac{\xi^{n+1 \choose 2}}{2^{n-1}}\right)^{2^{m+1}} \\
    &\leq \frac{2^{n-1}}{\xi^{n+1 \choose 2}} \exp\left(-\frac{\xi^{n+1 \choose 2}}{2^{n-1}}\right)^{2^{m+1}} \\
    &= \frac{2^{n-1}}{\xi^{n+1 \choose 2}} \exp\left(-2^{m-n+2} \xi^{n+1 \choose 2}\right) \\
    &\leq e^{-\sqrt{n}},
\end{align*}
where the last inequality is true for any
\begin{align*}
    m \geq n-2+\frac{1}{\ln(2)}\left(\frac{\sqrt{n}(n+1)}{2} + \ln\left(\sqrt{n} + \frac{\sqrt{n}(n+1)}{2} + (n-1)\ln(2)\right)\right).
\end{align*}
Hence, we can choose $m = O(n^\frac{3}{2})$. After applying \Cref{lemma: square} once again we need $m$ layers of $6$ enhanced neurons for this step. Finally, we calculate $x\frac{P(x)}{2^n}$ in the second to last layer and approximate $R(x)$ in the last layer. Since 
$x\frac{P(x)}{2^n} \leq 1$ for $x \in [-1,1]$ we calculated $R(x)$ up to an error of $e^{-\sqrt{n}}$. All in all, we need $O(n^\frac{3}{2})$ layers of $8$ neurons to approximate $\max\{0,x\}$ with an error of at most $\frac{5}{2}e^{-\sqrt{n}}$.
\end{proof}

At last, we have everything needed to approximate ReLU networks using a neural network with ReLU-computing activation functions. The idea is to simply replace every activation function by the network given in \Cref{lemma: approximate ReLU}. It is pivotal to consider how the errors are propagated through the network.

\begin{proof}[Proof of \Cref{theorem: ReLU-computing calculates relu}]
Our plan is to replace every ReLU-neuron by the network given in \Cref{lemma: approximate ReLU}. The inputs in the $k$th layer of the original network have an absolute value bounded by $(C(N+1))^{k-1}$. Thus, we need to scale by $(C(N+1))^{1-k}$ before applying \Cref{lemma: approximate ReLU}. Actually, we need to scale by a bit more because of the approximation error introduced in previous layers but this is negligible. \\

Since $\max\{0,x\}$ is Lipschitz continuous with constant $1$, the errors of each layer add up to the error of the entire network. Therefore, we can apply \Cref{lemma: approximate ReLU} with 
\begin{align*}
    n = \ln\left(\frac{L(C(N+1))^{k-1}}{\varepsilon}\right)^2
\end{align*}
to get an error of $\frac{\varepsilon}{L}$ in the $k$th layer. Note that we need $O(n^\frac{3}{2})$ layers in \Cref{lemma: approximate ReLU}. Together with
\begin{align*}
    \sum_{k=1}^L \ln\left(\frac{L(C(N+1))^{k-1}}{\varepsilon}\right)^3 = O(L\ln(\varepsilon^{-1}L)^3 + L^4\ln(CN)^3)
\end{align*}
we get the desired result.
\end{proof}

\newpage

\subsection{Approximating ReLU-computable Functions with ReLU}

This subsection deals with the dual problem of the last subsection. In lieu of approximating ReLU networks with ReLU-computing networks we intend to approximate ReLU-computable networks with ReLU networks. Once again, we start with the easier task of approximating the function $x \mapsto x^2$.

\begin{lemma}\label{lemma: affine approximation of square}
The function $x \mapsto x^2$ can be approximated in $[-1,1]$ by a neural network with $n+3$ layers of $3$ enhanced neurons up to an error of $4^{-n}$. We can choose an approximation, which is Lipschitz continuous with Lipschitz constant $2$ and constantly equal to $1$ outside of $[-1,1]$.
\end{lemma}

\begin{proof}
Define $h: \mathbb{R} \to \mathbb{R}$ by $h(x) = 2|x|-1$. We denote by $h^{(n)}$ the function $h$ applied $n$ times. Observe that for $x \in [-1,1]$ the functions $h^{(n)}$ alternates between $-1$ and $1$ at constant pace $2^{n-1}$ times. Let 
\begin{align*}
    I_{n,k} \coloneqq [k2^{-n}, (k+1)2^{-n}]
\end{align*}
and $L_{n,k}$ be the affine function connecting the points 
\begin{align*}
    (k2^{-n}, (k2^{-n})^2) \quad \text{and} \quad ((k+1)2^{-n}, ((k+1)2^{-n})^2).
\end{align*}
Define $g^{(0)}(x) = 1 + \frac{1}{2}(h(x)-1)$
\begin{align*}
    g^{(n+1)}(x) = g^{(n)}(x) + 2^{-2n-3} (h^{(n+2)}(x)-1).
\end{align*}

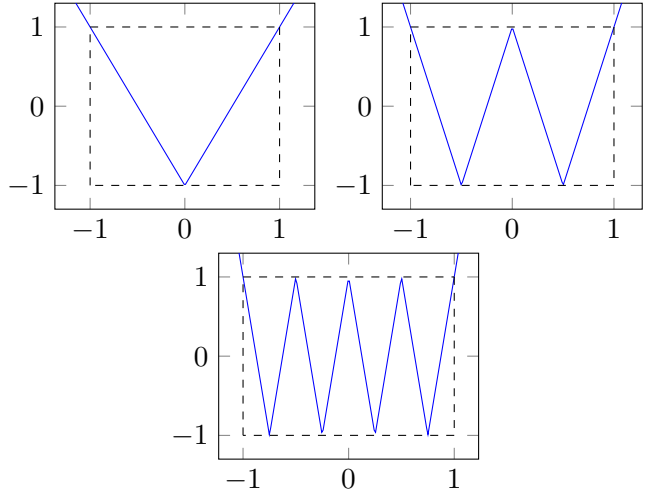
\begin{figure}[h]
    \centering
    \begin{tikzpicture}[
        evaluate={
            function f(\x) {
                return 2*abs(\x) - 1;
            };
        },
    ]
    \begin{axis}[no markers,ymin=-1.3,ymax=1.3]
    \addplot+[domain=-1.3:1.3,samples=200,color=blue]{f(x)};
    \addplot+[dashed,color=black] coordinates {(-1,-1)(-1,1)};
    \addplot+[dashed,color=black] coordinates {(-1,1)(1,1)};
    \addplot+[dashed,color=black] coordinates {(1,1)(1,-1)};
    \addplot+[dashed,color=black] coordinates {(1,-1)(-1,-1)};
    \end{axis}
    \end{tikzpicture}
    \begin{tikzpicture}[
        evaluate={
            function f(\x) {
                return 2*abs(\x) - 1;
            };
        },
    ]
    \begin{axis}[no markers,ymin=-1.3,ymax=1.3]
    \addplot+[domain=-1.3:1.3,samples=200,color=blue]{f(f(x))};
    \addplot+[dashed,color=black] coordinates {(-1,-1)(-1,1)};
    \addplot+[dashed,color=black] coordinates {(-1,1)(1,1)};
    \addplot+[dashed,color=black] coordinates {(1,1)(1,-1)};
    \addplot+[dashed,color=black] coordinates {(1,-1)(-1,-1)};
    \end{axis}
    \end{tikzpicture}
    \begin{tikzpicture}[
        evaluate={
            function f(\x) {
                return 2*abs(\x) - 1;
            };
        },
    ]
    \begin{axis}[no markers,ymin=-1.3,ymax=1.3]
    \addplot+[domain=-1.3:1.3,samples=200,color=blue]{f(f(f(x)))};
    \addplot+[dashed,color=black] coordinates {(-1,-1)(-1,1)};
    \addplot+[dashed,color=black] coordinates {(-1,1)(1,1)};
    \addplot+[dashed,color=black] coordinates {(1,1)(1,-1)};
    \addplot+[dashed,color=black] coordinates {(1,-1)(-1,-1)};
    \end{axis}
    \end{tikzpicture}
    \caption{Graphs of $h^{(n)}$ for $n=1,2,3$}
\end{figure}

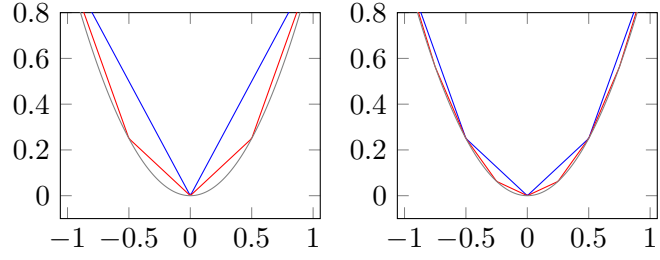
\begin{figure}[h]
    \centering
    \begin{tikzpicture}[
        evaluate={
            function f(\x) {
                return 2*abs(\x) - 1;
            };
        },
    ]
    \begin{axis}[no markers,ymin=-0.1,ymax=0.8]
    \addplot+[domain=-1.3:1.3,samples=200,color=blue]{1+1/2*(f(x)-1)};
    \addplot+[domain=-1.3:1.3,samples=200,color=red]{1+1/2*(f(x)-1)+1/8*(f(f(x))-1)};
    \addplot+[domain=-1.3:1.3,samples=200,color=gray]{x^2};
    \end{axis}
    \end{tikzpicture}
    \begin{tikzpicture}[
        evaluate={
            function f(\x) {
                return 2*abs(\x) - 1;
            };
        },
    ]
    \begin{axis}[no markers,ymin=-0.1,ymax=0.8]
    \addplot+[domain=-1.3:1.3,samples=200,color=blue]{1+1/2*(f(x)-1)+1/8*(f(f(x))-1)};
    \addplot+[domain=-1.3:1.3,samples=200,color=red]{1+1/2*(f(x)-1)+1/8*(f(f(x))-1)+1/32*(f(f(x))-1)};
    \addplot+[domain=-1.3:1.3,samples=200,color=gray]{x^2};
    \end{axis}
    \end{tikzpicture}
    \caption{Construction of $g^{(n)}$}
\end{figure}

We would like to prove that $g^{(n)}|_{I_{n,k}} = L_{n,k}$ for $k \in \{-2^n,\dots,2^n-1\}$ by induction on $n$. The base case is obvious. Assume that we have already proved the claim for some $n$ and any $k \in \{-2^n,\dots,2^n-1\}$. Since $g^{(n)}$ and $h^{(n+2)}$ are affine in $I_{n+1,2k}$ and $I_{n+1,2k+1}$, the same holds true for $g^{(n+1)}$. Thus, it suffices to calculate the value of $g^{(n+1)}$ at the endpoints of those two intervals. Because of
\begin{align*}
    h^{(n+2)}(2k2^{-n-1}) = h^{(n+2)}((2k+2)2^{-n-1}) = 1
\end{align*}
we have
\begin{align*}
    g^{(n+1)}(2k2^{-n-1}) &= g^{(n)}(k2^{-n}) = (k2^{-n})^2 \quad \text{and} \\
    g^{(n+1)}((2k+2)2^{-n-1}) &= g^{(n)}((k+1)2^{-n}) = ((k+1)2^{-n})^2.
\end{align*}
The other two endpoints are both at $(2k+1)2^{-n-1}$ and we get
\begin{align*}
    g^{(n)}((2k+1)2^{-n-1}) &= \frac{1}{2}(g^{(n)}(k2^{-n}) + g^{(n)}((k+1)2^{-n})) \\
    &= \left(k^2+k+\frac{1}{2}\right)4^{-n} \quad \text{and} \\
    h^{(n+2)}((2k+1)2^{-n-1}) &= -1 \\
    \Rightarrow g^{(n+1)}((2k+1)2^{-n-1}) &= \left(k^2+k+\frac{1}{2}\right)4^{-n} + 2^{-2n-3} (-2) \\
    &= ((2k+1)2^{-n-1})^2,
\end{align*}
which concludes the induction step. Next up, we would like to show that $g^{(n)}$ is a good approximation for $x^2$. Observe that for $x \in [k2^{-n}, (k+1)2^{-n}]$ we can write $x = k2^{-n} + \delta$ with $\delta \in [0,2^{-n}]$ and get
\begin{align*}
    |g(x) - x^2| &= |2^{-n}(2k+1) (x-k2^{-n}) + k^2 4^{-n} - x^2| \\
    &= |2^{-n}(2k+1) \delta + k^2 4^{-n} - (k2^{-n} + \delta)^2| \\
    &= |\delta (2^{-n} - \delta)| \\
    &\leq 4^{-n}.
\end{align*}
Furthermore, $g^{(n)}$ is Lipschitz continuous with constant $2$. It remains to set up the neural network and make sure that that the function is equal to $1$ outside of $[-1,1]$. We use $2$ enhanced neurons in the $m$th layer to iteratively calculate $h^{(m)}$. In the third enhanced neuron we are calculating $g^{(m-2)}$ using
\begin{align*}
    g^{(m-2)} = \max\{0, g^{(m-3)}(x) + 2^{-2m+3}(h^{(m-1)}(x)-1) + 1\} - 1.
\end{align*}
In the last layer we compose $g^{(n)}$ with the function
\begin{align*}
    k(x) = 
    \begin{cases}
        x &\text{for } x \in (-\infty,1], \\
        1 &\text{otherwise}
    \end{cases}
\end{align*}
to get an approximation of $x \mapsto x^2$ with the desired properties.
\end{proof}

As we have already seen in \Cref{lemma: product}, an approximate square can be used to multiply two numbers. This in turn allows us to approximate general polynomials with ReLU networks.

\begin{lemma}\label{lemma: affine approximation of polynomial}
Let $\varepsilon > 0$. Any polynomial of degree $m$ with coefficients bounded by $C$ can be approximated in $[-1,1]$ by a neural network with $O(m \ln(2C\varepsilon^{-1}m))$ layers of $8$ neurons with ReLU activation function up to an error of $\varepsilon$.
\end{lemma}

\begin{proof}
We use $1$ neuron in each layer to save $x$ with the help of
\begin{align*}
    x = \max\{0,x+1\}-1
\end{align*}
for $x \in [-1,1]$. Using \Cref{lemma: product} and \Cref{lemma: affine approximation of square} we can iteratively calculate $x^k$ from $x^{k-1}$ and $x$ using $6$ neurons per layer. We use the final neuron in each layer to save our intermediate result, which consists of the summands up until $x^k$. This can be done with only $1$ neuron because the intermediate results are uniformly bounded below in $[-1,1]$ by some constant. \\

It remains to calculate the error. The error in one application of \Cref{lemma: product} is the sum of the errors of the two squares used in that lemma. Since $x \mapsto xy$ is Lipschitz continuous with constant $1$ for $y \in [-1,1]$, those errors just sum up to the final error in $x^k$. Assume that this error is bounded by $\frac{\varepsilon}{Cm}$. Then the overall error in the polynomial is bounded by $\varepsilon$, which is exactly what we want. \\

Note that we need $m-1$ multiplications to get to $x^m$ and each of them needs two squares. Hence, we need to calculate the squares with a precision of $\frac{\varepsilon}{2Cm(m-1)}$. In order to do this we choose 
\begin{align*}
    n = O\left[\ln \left(\frac{2Cm(m-1)}{\varepsilon}\right)\right]
\end{align*}
in \Cref{lemma: affine approximation of square}. All in all, we need $O(mn)$ layers, which can be expressed as $O(m \ln(2C\varepsilon^{-1}m))$.
\end{proof}

Before we dive into the next part of the proof let us ponder on \Cref{def: ReLU-computable}, where ReLU-computable activation functions were defined. The inequality in the second condition implies that $\rho$ is Lipschitz continuous with constant $1$. It is worth noting that we can relax the second condition a bit. For any finite set $F \subset \mathbb{R}$ it suffices if $\rho$ satisfies the condition on $\mathbb{R} \setminus F$. All of the results in this section will still be true but it would make the proofs a bit harder to read. \\

A lot of the commonly used activation functions are ReLU-computable (in the slightly more general sense). At some point Fa\`a di Bruno's formula will come in handy, which is stated in the following.
\begin{theorem}[\cite{faaDiBruno}, Introduction]
\label{theorem: faaDiBruno}
Let $f$ and $g$ be $n$ times differentiable functions. Then we have
\begin{align*}
    D^n (f \circ g) = \sum \frac{n!}{k_1!\cdots k_n!} (D^{k_1 + \dots + k_n} f \circ g) \prod_{m=1}^n \left(\frac{D^m g}{m!}\right)^{k_m},
\end{align*}
where the sum runs over all $n$-tuples $(k_1,\dots,k_n)$ satisfying
\begin{align*}
    \sum_{i=1}^n ik_i = n.
\end{align*}
\end{theorem}
Here are some examples of ReLU-computable activation functions:

\begin{itemize}
    \item Gaussian ($e^{-x^2}$): The first property is obvious, so we focus on the second one. We have
    \begin{align*}
        \frac{\mathrm{d}^n}{\mathrm{d}x^n} e^{-x^2} = H_n(x) e^{-x^2},
    \end{align*}
    where $H_n$ denotes the so-called $n$th Hermite polynomial. By \Cref{theorem: faaDiBruno} we have
    \begin{align*}
        H_{2m}(x) &= (2m)! \sum_{l=0}^m \frac{(-1)^{m-l}}{(2l)!(m-l)!} (2x)^{2l} \quad \text{and} \\
        H_{2m+1}(x) &= (2m+1)! \sum_{l=0}^m \frac{(-1)^{m-l}}{(2l+1)!(m-l)!} (2x)^{2l+1}.
    \end{align*}
    Observe that $|x^n e^{-x^2}|$ has global maxima at $x = \pm \sqrt{\frac{n}{2}}$ and
    \begin{align*}
        |x^n e^{-x^2}| \leq \left(\frac{n}{2}\right)^\frac{n}{2} e^{-\frac{n}{2}}.
    \end{align*}
    This implies
    \begin{align*}
        \frac{1}{(2m)!}\left|\frac{\mathrm{d}^{2m}}{\mathrm{d}x^{2m}} e^{-x^2}\right| &= \left|\sum_{l=0}^m \frac{(-1)^{m-l}}{(2l)!(m-l)!} (2x)^{2l} e^{-x^2}\right| \\
        &\leq \sum_{l=0}^m \frac{1}{(2l)!(m-l)!} 2^{2l} l^l e^{-l} \\
        &= \sum_{l=0}^m \frac{\left(\frac{2l}{e}\right)^{2l}}{(2l)!} \frac{1}{\left(\frac{l}{e}\right)^l} \frac{1}{(m-l)!} \\
        &\leq \sum_{l=0}^m 1 \cdot \frac{1}{l!} \cdot  \frac{1}{(m-l)!} \\
        &= \frac{1}{m!} \sum_{l=0}^m {m \choose l} \\
        &= \frac{2^m}{m!},
    \end{align*}
    where the second inequality follows from Stirling. The estimate for $2m+1$ works similarly.

    \item Logistic $\left(\frac{1}{1+e^{-x}}\right)$: Once again, the first property is immediate to see. For the second property we pick $\frac{3\pi}{4} < r < \pi$. By Cauchy's integral formula we have
    \begin{align*}
        \frac{1}{n!} \frac{\mathrm{d}^n}{\mathrm{d}x^n} \frac{e^x}{1+e^x} &= \frac{1}{2\pi i} \oint_{\partial B(x,r)} \frac{e^z}{(1+e^z)(z-x)^{n+1}} \, \mathrm{d}z \\
        &= \frac{1}{2\pi i} \int_0^{2\pi} \frac{e^{x+re^{it}}}{(1+e^{x+re^{it}}) (re^{it})^{n+1}} \cdot rie^{it} \, \mathrm{d}t
    \end{align*}
    since $\frac{e^x}{1+e^x}$ is holomorphic in $\{z \in \mathbb{C}: |\I(z)| < \pi\}$. This implies
    \begin{align*}
        \frac{1}{n!} \left|\frac{\mathrm{d}^n}{\mathrm{d}x^n} \frac{e^x}{1+e^x}\right| \leq \frac{1}{2\pi} \int_0^{2\pi} \left|\frac{e^{x+re^{it}}}{1+e^{x+re^{it}}}\right| r^{-n} \, \mathrm{d}t.
    \end{align*}
    Note that
    \begin{align*}
        |1+e^z| = |1 + \cos(\I(z)) e^{\R(z)} + i \sin(\I(z)) e^{\R(z)}|
    \end{align*}
    For $|\I(z)| \leq \frac{\pi}{4}$ we have $\cos(\I(z)) \geq \frac{1}{\sqrt{2}}$ and therefore
    \begin{align*}
        |1+e^z| \geq 1 + \cos(\I(z)) e^{\R(z)} \geq \frac{1}{\sqrt{2}} e^{\R(z)} \geq \sin(r) e^{\R(z)}.
    \end{align*}
    For $\frac{\pi}{4} \leq |\I(z)| \leq r$ we have $|\sin(\I(z))| \geq \sin(r)$ and thus
    \begin{align*}
        |1+e^z| \geq \sin(\I(z)) e^{\R(z)} \geq \sin(r) e^{\R(z)}.
    \end{align*}
    Together with $|e^z| = e^{\R(z)}$ we can conclude that
    \begin{align*}
        \frac{1}{2\pi} \int_0^{2\pi} \left|\frac{e^{x+re^{it}}}{1+e^{x+re^{it}}}\right| r^{-n} \, \mathrm{d}t &\leq \frac{1}{2\pi} \int_0^{2\pi} \frac{1}{\sin(r)} r^{-n} \, \mathrm{d}t \\
        &= \frac{1}{\sin(r) r^n}.
    \end{align*}

    \item $\tanh(x)$: Observe that
    \begin{align*}
        \tanh(x) = \frac{2}{1+e^{-2x}} - 1.
    \end{align*}
    The estimate of the Logistic function gives
    \begin{align*}
        \frac{1}{n!} \left|\frac{\mathrm{d}^n}{\mathrm{d}x^n} \tanh(x)\right| \leq \frac{2^{n+1}}{\sin(r)r^n}.
    \end{align*}
    With $r = \frac{5\pi}{6}$ we get
    \begin{align*}
        \frac{2^{n+1}}{\sin(r)r^n} &= \frac{2^{n+1}}{\frac{1}{2} \left(\frac{5\pi}{6}\right)^n} \\
        &= 4\left(\frac{12}{5\pi}\right)^n.
    \end{align*}
    Hence, almost all of the properties are immediate consequences of the fact that the Logistic function is ReLU-computable. It only remains to check that
    \begin{align*}
        \frac{1}{n!} \left|\frac{\mathrm{d}^n}{\mathrm{d}x^n} \tanh(x)\right| \leq 1
    \end{align*}
    for $n = 1,\dots,5$, which is a matter of calculation.
    
    \item Softplus ($\ln(1+e^x)$): For the first property we use the well-known inequality
    \begin{align*}
        \ln(t) \leq t - 1.
    \end{align*}
    We can conclude that
    \begin{align*}
        0 &\leq \ln(1+e^x) - x \\
        &= \ln(1+e^{-x}) \\
        &\leq e^{-x}
    \end{align*}
    as well as
    \begin{align*}
        0 &\leq \ln(1+e^x) \\
        &\leq e^x
    \end{align*}
    and therefore
    \begin{align*}
        \ln(1+e^x) - \max\{0,x\} \xrightarrow{|x| \to \infty} 0
    \end{align*}
    exponentially quickly. The second property follows from the fact that the derivative of the Softplus function is the Logistic function.

    \item ELU ($e^x-1$ for $x \leq 0$ and $x$ for $x > 0$): As already mentioned, one can relax the definition of a ReLU-computable function by allowing finitely many exception points where the function is not smooth. All of the proofs still work out. We opted not to do that here to make the proofs easier to read. However, with this relaxed definition ELU is also ReLU-computable.

    \item Leaky ReLU ($0.01x$ for $x \leq 0$ and $x$ for $x > 0$): Again, this function would satisfy a relaxed version of the definition of a ReLU-computable function. But for this activation function there are much simpler ways to prove the results of this section.
\end{itemize}

By now we have seen that there are lots of ReLU-computable activation functions. But why are they useful? The answer is that ReLU networks can approximate ReLU-computable functions, as we will show in the following.

\begin{lemma}\label{lemma: affine approximation of activation function}
Let $\rho: \mathbb{R} \to \mathbb{R}$ be ReLU-computable. A neural network of depth $O(\ln(\varepsilon^{-1})^3)$ and width $11$ with ReLU activation function can approximate $\rho$ with an error of at most $\varepsilon$.
\end{lemma}

\begin{proof}
Let $0 < \varepsilon < \frac{1}{3}$. Since $\rho$ is ReLU-computable, there is a piecewise affine function $l: \mathbb{R} \to \mathbb{R}$ such that $|\rho(x)-l(x)|$ decays exponentially quickly. Hence, there is a compact interval $I$ of size $O(\ln(\varepsilon^{-1}))$ such that $|\rho(x)-l(x)| < \varepsilon$ outside of that interval. We can cover $I$ by $O(\ln(\varepsilon^{-1}))$ intervals of size $2$ such that neighboring intervals overlap by exactly $\varepsilon$. For each of those intervals $J$ we consider the Taylor expansion around its midpoint $a_J$. The $m$th Taylor polynomial $P_m$ satisfies
\begin{align*}
    |\rho(x)-P_m(x)| &= \left|\frac{\rho^{(m+1)}(\xi)}{m!} (x-a_J)^{m+1}\right| \\
    &\leq \left|\frac{\rho^{(m+1)}(\xi)}{m!}\right|
\end{align*}
for some $\xi \in J$. Since $\rho$ is ReLU-computable, we know 
that $\left|\frac{\rho^{(m+1)}(\xi)}{m!}\right|$ converges to $0$ exponentially quickly for $n \to \infty$. Therefore, we can pick $m = O(\ln(\varepsilon^{-1}))$ to approximate $\rho$ in $J$ with an error of at most $\varepsilon$. \\

We proceed with the design of the neural network. We apply \Cref{lemma: affine approximation of polynomial} to calculate those Taylor polynomials with the help of $O(m \ln(\varepsilon^{-1}m))$ layers of $8$ neurons. Here we used that the coefficients of the Taylor polynomial are bounded by $1$ since $\rho$ is ReLU-computable. In addition to that, we use $2$ neurons in each layer to save the value of $x$. \\

After calculating one of those polynomials, we calculate a translated version of the function
\begin{align*}
    x \mapsto 
    \begin{cases}
        \frac{1+x}{\varepsilon} &\text{for } x \in [-1,-1+\varepsilon], \\
        1 &\text{for } x \in [-1+\varepsilon,1-\varepsilon], \\
        \frac{1-x}{\varepsilon} &\text{for } x \in [1-\varepsilon, 1], \\
        0 &\text{otherwise}
    \end{cases}
\end{align*}
and multiply it with the polynomial using \Cref{lemma: product} and \Cref{lemma: affine approximation of square}. This gives us a partition of unity, so that we get a good approximation for $f(x)$ in the overlapping regions as well. Note that multiplying with $0$ using \Cref{lemma: product} always gives exactly $0$. The product is then added to our intermediate result, which needs $1$ neuron in each layer to be stored because it remains bounded for now. In the end we calculate $l(x)$ and add it to our intermediate result to get the final approximation. All in all, this needs $O(\ln(\varepsilon^{-1})^3)$ layers with $11$ neurons each.
\end{proof}

Finally, we can put all the pieces together to prove that a ReLU networks can approximate neural networks with ReLU-computable activation functions.

\begin{proof}[Proof of \Cref{theorem: relu calculates ReLU-computable}]
Similar to \Cref{theorem: ReLU-computing calculates relu} we would like to replace all of the activation functions by the blocks given in \Cref{lemma: affine approximation of activation function}. Each layer multiplies the error of the previous layers by a factor of at most $C(N+1)$ and adds its own error. Since $\rho$ is Lipschitz continuous with constant $1$, the error of the entire network is given by the sum over $k$ of $(C(N+1))^{L-k}$ times the error of the $k$th layer. Therefore, we can apply \Cref{lemma: affine approximation of activation function} with 
\begin{align*}
    \varepsilon_k \coloneqq \frac{\varepsilon}{L(C(N+1))^{L-k}}
\end{align*}
to get an overall error of $\varepsilon$. Note that we need $O(\ln(\varepsilon_k^{-1})^3)$ layers in \Cref{lemma: affine approximation of activation function}. Together with
\begin{align*}
    \sum_{k=1}^L \ln(\varepsilon_k^{-1})^3 &= \sum_{k=1}^L (\ln(L\varepsilon^{-1}) + (L-k)\ln(C(N+1)))^3 \\
    &= O(L\ln(L\varepsilon^{-1})^3 + L^4\ln(CN)^3)
\end{align*}
we get the desired result.
\end{proof}

\newpage

\section{Exporting Results to Other Activation Functions}

In the last section we showed that on the one hand ReLU networks can be approximated by networks with ReLU-computing activation functions. On the other hand ReLU networks can approximate networks with ReLU-computable activation functions. Thus, we can use the following framework to export results for ReLU networks to the more general classes of networks with ReLU-computing and ReLU-computable activation functions:

\begin{enumerate}
    \item Approximate with a ReLU-network using \Cref{theorem: relu calculates ReLU-computable}.
    \item Use the result for ReLU networks.
    \item Approximate the resulting network with \Cref{theorem: ReLU-computing calculates relu}.
\end{enumerate}

An example for the application of this framework is the following width inefficiency result.

\begin{proof}[Proof of \Cref{theorem: width inefficiency for ReLU-computable activation}]
Assume that this is not true. Consider the function $f$ from \Cref{theorem: width inefficiency in 1 dimension} with $L^{10}$ instead of $L$. Since we are doing a proof by contradiction, we use the framework in another order.
\begin{enumerate}
    \item Using \Cref{theorem: ReLU-computing calculates relu} we can approximate $f$ with the help of a network of depth $O(L^{40})$, width $24$ and activation function $\rho$.
    \item By assumption this can be approximated by a neural network with activation function $\rho$ of depth $L$ and subexponential width and coefficients.
    \item By \Cref{theorem: relu calculates ReLU-computable} this can in turn be approximated by a ReLU network of depth $O(L^4\ln(CN)^3)$ and subexponential width.
    \item However, this is a contradiction to \Cref{theorem: width inefficiency in 1 dimension} because any network of depth $\leq L^5$ with this property needs exponential width.
\end{enumerate}
\end{proof}

The exponent of $L^{40}$ is far from optimal. The point of \Cref{theorem: width inefficiency for ReLU-computable activation} is that it is still subexponential, which is a significant difference to \Cref{theorem: depth efficiency}. \\

Another observation is that we could apply \Cref{theorem: ReLU-computing calculates relu} because we knew that the coefficients in \Cref{theorem: width inefficiency in 1 dimension} are bounded by $6$. Furthermore, we did not need to control the coefficients in \Cref{theorem: relu calculates ReLU-computable} or \Cref{theorem: ReLU-computing calculates relu}. Unfortunately, this might be necessary for other results like the depth efficiency result in \Cref{theorem: depth efficiency}. \\

The key takeaway is the following: If one can control the coefficients, then a result that is true for ReLU networks is also true for other activation functions.

\newpage

\section*{Acknowledgement}
\addcontentsline{toc}{section}{Acknowledgement}

I am very grateful to my advisor Christoph Thiele for suggesting this fascinating subject and for many helpful discussions. Furthermore, I would like to thank Johannes Linn and Lars Becker for reading over this thesis.

\newpage
\emergencystretch=2em
\printbibliography[heading = bibintoc]

\end{document}